\documentclass[a4paper,article,onecolumn,english]{memoir}

\usepackage[utf8]{inputenc}
\usepackage{babel}
\usepackage{lipsum}
\usepackage[usenames,x11names]{xcolor}
\usepackage[font=normalsize]{caption}
\usepackage{pdflscape}
\usepackage{multirow}
\usepackage{enumitem}
\usepackage{amsmath}
\usepackage{amsfonts}
\usepackage{amssymb}
\usepackage{amsthm}
\usepackage{thmtools}
\usepackage{dsfont}
\usepackage[left=2.5cm,right=2.5cm,top=2.5cm,bottom=2.5cm]{geometry}
\usepackage[colorlinks=true,
            linkcolor = blue,
            urlcolor  = blue,
            citecolor = red,
			breaklinks]{hyperref}

\numberwithin{equation}{chapter}

\newcommand{\eps}{\varepsilon}
\newcommand{\bsy}[1]{\boldsymbol{#1}}

\newcommand{\mb}[1]{\mathbf{#1}}

\newcommand{\Xox}{X_0^{x\smash[t]{\mathstrut}}}
\newcommand{\Xtx}{X_t^{x\smash[t]{\mathstrut}}}
\newcommand{\Xsx}{X_s^{x\smash[t]{\mathstrut}}}
\newcommand{\Xlx}{X_\ell^{x\smash[t]{\mathstrut}}}

\newcommand{\wXtx}{\widetilde{X}_t^{x\smash[t]{\mathstrut}}}

\newtheoremstyle{slplain}
  {0.4cm}
  {0.4cm}
  {\upshape}
  {}
  {\bfseries}
  {.}
  { }
  {}
  
\newtheoremstyle{itplain}
    {0.4cm}
    {0.4cm}
    {\itshape}
    {}
    {\bfseries}
    {.}
    { }
    {}
  
\declaretheorem[style=slplain,numberwithin=chapter]{definition}
\declaretheorem[style=slplain,sibling=definition]{example}
\declaretheorem[style=slplain,sibling=definition]{remark}

\declaretheorem[style=itplain,sibling=definition]{theorem}

\declaretheorem[style=itplain,sibling=definition]{proposition}
\declaretheorem[style=itplain,sibling=definition]{lemma}

\setsecheadstyle{\large\bfseries\itshape}
\setsubsecheadstyle{\normalsize\bfseries\itshape}

\AtBeginDocument{
	\addtolength{\abovedisplayskip}{0.125ex}
	\addtolength{\abovedisplayshortskip}{0.125ex}
	\addtolength{\belowdisplayskip}{0.125ex}
	\addtolength{\belowdisplayshortskip}{0.125ex}
}

\let\OLDthebibliography\thebibliography
\renewcommand\thebibliography[1]{
  \OLDthebibliography{#1}
  \setlength{\parskip}{0pt}
  \setlength{\itemsep}{3pt plus 0.3ex}
}

\definecolor{shadecolor}{gray}{0.96}

\title{\bfseries The spectral expansion approach to index transforms and connections with the theory of diffusion processes\\[0.3cm]}

\author{Rúben Sousa
\thanks{Corresponding author. CMUP, Departamento de Matemática, Faculdade de Ciências, Universidade do Porto, Rua do Campo Alegre 687, 4169-007 Porto, Portugal. Email: \texttt{ruben.sousa@outlook.com}}
\and
Semyon Yakubovich \thanks{Departamento de Matemática, Faculdade de Ciências, Universidade do Porto, Rua do Campo Alegre 687, 4169-007 Porto, Portugal. Email: \texttt{syakubov@fc.up.pt}}
\\[0.3cm]
}
\date{June 16, 2017\\[0.7cm]}
\renewcommand{\maketitlehookd}{%
  \par {\centering
   \hrule
     \begin{abstract}
     \mbox{  }\\[-8mm]
Many important index transforms can be constructed via the spectral theory of Sturm-Liouville differential operators. Using the spectral expansion method, we investigate the general connection between the index transforms and the associated parabolic partial differential equations. We show that the notion of Yor integral, recently introduced by the second author, can be extended to the class of Sturm-Liouville integral transforms.

We furthermore show that, by means of the Feynman-Kac theorem, index transforms can be used for studying Markovian diffusion processes. This gives rise to new applications of index transforms to problems in mathematical finance. \\[1.2mm]
     
     \textbf{Keywords:} Sturm-Liouville spectral problem, index transforms, Yor integral, parabolic equations, diffusion processes
     \end{abstract}
    \hrule
    \mbox{   }\\[-0.1cm]
  \par}}

\begin{document}

\maketitle

\begingroup
\let\clearforchapter\relax
\chapter{Introduction}
\endgroup

Index transforms are integral transforms whose kernel depends on the parameters (or indices) of well-known special functions \cite{yakubovich1996}. Despite being less well known than the classical Fourier, Laplace and Mellin transforms, various index transforms, such as the Kontorovich-Lebedev and the Mehler-Fock transform, have found important applications to problems arising in physics (e.g.\ \cite{gasaneo2001, jungpedersen2012}) and, more recently, in mathematical finance \cite{craddock2015, linetsky2006}.

It is well-known that many different integral transforms can be deduced by carrying out a spectral expansion with respect to the eigenfunctions of the corresponding second-order Sturm-Liouville differential operators \cite{dunfordschwartz1963, srivastava2000, titchmarsh1962}. In particular, the Kontorovich-Lebedev and the (ordinary) Mehler-Fock index transforms have been derived by applying this technique to differential operators related with the modified Bessel equation and the Legendre equation, respectively (see \cite{titchmarsh1962} and \cite{neretin2001} respectively). On the other hand, such spectral representation techniques have also been successfully applied to a large family of parabolic partial differential equation (PDE) problems and, as a by-product, to the characterization of the (Markovian diffusion) stochastic processes whose infinitesimal generator is a second-order Sturm-Liouville operator (see \cite{linetsky2006} and the references therein).

However, the connection between the index (and other integral) transforms, the parabolic PDEs and the associated diffusion processes is still rather unexplored in the literature. The heat kernels of the PDEs associated with some index transforms have been recently investigated in \cite{yakubovich2011} and \cite{rodrigues2013}, but in a non-unified fashion, and without reference to the related diffusion processes. In parallel, a general discussion of this connection is also lacking in the literature on spectral methods applied to problems modeled by diffusion processes. (It was pointed out in \cite{linetsky2004} that the index Whittaker transform is related to the basic diffusion process for the modeling of Asian options in mathematical finance, but this is just a particular case of the general connection.)

The aim of this paper is to demonstrate that the spectral expansion technique is an effective tool for studying integral transforms of the index type, not only because it yields a systematic procedure for deriving the index transforms, but also, and more importantly, because it gives rise (under a very general framework) to an explicit integral representation for the fundamental solution of the associated parabolic PDE, which is also the transition probability density of a Markovian diffusion process. Moreover, we will show that the spectral approach allows us to provide a natural generalization of the so-called Yor integral, which was introduced in \cite{yakubovich2013} under the particular context of the Kontorovich-Lebedev transform. As we shall see, the Yor integral can be extended to the whole family of integral transforms arising from Sturm-Liouville operators, and from this extension a general relation between the (inverse) transforms and the associated PDEs is obtained.

Valuation problems in mathematical finance constitute the main motivation behind the growing literature on spectral-theoretic approaches to the study of diffusion processes. Our study exhibits the general nature of the applications of index transforms to mathematical finance, and provides a framework for further investigations.

This paper is structured as follows. In Section \ref{chap:preliminaries} we set up some notation and review the relevant background material on index transforms and special functions. Section \ref{chap:spectral} starts by presenting the general spectral theory of Sturm-Liouville differential operators, the resulting generalized Fourier transforms and the connection with the associated parabolic PDE. Then, in Subsection \ref{sec:index_examples}, we treat in detail the Sturm-Liouville operators which yield three important index transforms, namely the Kontorovich-Lebedev, index Whittaker and Mehler-Fock transforms. In Subsection \ref{sec:yorintegrals} we introduce the generalized Yor integral as the inverse generalized Fourier transform of an exponential function and we study some of its properties, again focusing on the Yor integrals which result from index transforms. Finally, Section \ref{chap:diffusion} is devoted to the link with diffusion processes: we describe the construction of the diffusion process which generates a given Sturm-Liouville operator, and we show that the famous Feynman-Kac theorem can be used to derive some interesting properties of index transforms and the corresponding Yor integrals.

\chapter{Preliminaries} \label{chap:preliminaries}

Throughout this work we shall denote by $L_2\bigl(\Omega; w(x) dx\bigr)$ the weighted $L_2$-space with norm
\[
\|f\|_{L_2(\Omega; w(x) dx)} = \biggl(\int_{\Omega} |f(x)|^2\, w(x) dx \biggr)^{1/2}.
\]

We say that $\rho = \{\rho_{ij}\}_{i,j=1,2}$ is a \textit{positive $2 \times 2$ matrix measure} (\cite{dunfordschwartz1963}, XIII.5.6) if each $\rho_{ij}$ is a complex-valued set function defined on the Borel subsets of $\mathbb{R}$, and the following conditions hold:\vspace{-2pt}
\begin{enumerate}\setlength{\itemsep}{-1pt}
\item[(i)] The matrix $\rho(B)$ is Hermitian and positive semi-definite for every bounded Borel set $B$;
\item[(ii)] $\rho_{ij}(\bigcup_{k=1}^\infty B_k) = \sum_{k=1}^\infty \rho_{ij}(B_k)$ for any disjoint Borel sets $B_1, B_2, \ldots$ with bounded union.
\end{enumerate}
If, in addition, we have $\rho(I) = 0$ for any interval $I \subseteq (-\infty, 0)$, we shall say that $\rho$ is a \textit{positive matrix measure on the nonnegative real axis}. For such a measure, we denote by $L_2\bigl([0,\infty); \rho\bigr)$ the Hilbert space obtained by completion of the space of all bounded, piecewise continuous, compactly supported functions $\psi(\lambda) = (\psi_1(\lambda), \psi_2(\lambda))$ with respect to the inner product
\[
(\psi, \eta) = \int_{0-}^\infty \sum_{j,k=1}^2 \psi_j(\lambda)\, \overline{\eta_k(\lambda)}\, d\rho_{jk}(\lambda).
\]

The \textit{Kontorovich-Lebedev transform} is defined by \cite{kontorovich1938, sneddon1972, yakubovich1996}
\begin{equation} \label{eq:prel_kl_transf}
K[f](\tau) = \int_0^\infty K_{i\tau}(y) f(y)\, {dy \over y}
\end{equation}
where the integral converges with respect to the norm in $L_2\bigl((0,\infty); \tau \sinh(\pi\tau) d\tau\bigr)$. Here $K_{i\tau}(y)$ is the modified Bessel function of the second kind with purely imaginary index $i\tau$. This integral transform is an isometric isomorphism \cite{yakubovich2004}
\[
K[\cdot]: L_2\bigl((0,\infty); \tfrac{dy}{y}\bigr) \, \to \, L_2\bigl((0,\infty); \tau \sinh(\pi\tau) d\tau\bigr)
\]
which yields the Parseval identity
\[
{2 \over \pi^2} \int_0^\infty \tau \sinh(\pi\tau) |K[f](\tau)|^2 d\tau = \int_0^\infty |f(y)|^2 {dy \over y}
\]
and whose inverse operator is defined by the formula
\begin{equation} \label{eq:prel_kl_inverse}
f(y) = {2 \over \pi^2} \int_0^\infty \tau \sinh(\pi\tau) K_{i\tau}(y) K[f](\tau) \, d\tau
\end{equation}
where the integral converges with respect to the norm in $L_2\bigl((0,\infty); {dy \over y}\bigr)$.

The so-called \textit{index Whittaker transform} is the integral transform \cite{srivastava1998, wimp1964}
\begin{equation} \label{eq:prel_iw_transf}
W_\alpha[f](\tau) = \int_0^\infty W_{\alpha, i\tau}(y) f(y) {dy \over y^2}
\end{equation}
where $\alpha < {1 \over 2}$ is a real parameter, and the integral converges with respect to the norm of the space $L_2\bigl((0,\infty);\, \tau \sinh(2\pi\tau)\, |\Gamma({1 \over 2} - \alpha +i\tau)|^2 d\tau\bigr)$. (Here $\Gamma(\cdot)$ is the Gamma function, cf.\ \cite{lebedev1965}.) The function $W_{\alpha, i\tau}(x)$ is the Whittaker function with indices $\alpha < {1 \over 2}$ and $i\tau \in [0,i\infty)$. The index Whittaker transform is an isometric isomorphism
\[
W_\alpha[\cdot]: L_2\bigl((0,\infty); \tfrac{dy}{y^2}\bigr) \, \to \, L_2\bigl((0,\infty);\, \tau \sinh(2\pi\tau)\, \bigl|\Gamma\bigl(\tfrac{1}{2} - \alpha +i\tau\bigr)\bigr|^2 d\tau\bigr)
\]
The corresponding Parseval identity is
\[
{1 \over \pi^2} \int_0^\infty \tau \sinh(2\pi\tau) \bigl|\Gamma\bigl(\tfrac{1}{2} - \alpha +i\tau\bigr)\bigr|^2 |W_\alpha[f](\tau)|^2 d\tau = \int_0^\infty |f(y)|^2 {dy \over y^2}
\]
and the inverse operator is given by
\begin{equation} \label{eq:prel_iw_inverse}
f(y) = {1 \over \pi^2} \int_0^\infty \tau \sinh(2\pi\tau) \bigl|\Gamma\bigl(\tfrac{1}{2} - \alpha +i\tau\bigr)\bigr|^2 W_{\alpha,i\tau}(y) W_\alpha[f](\tau) \, d\tau
\end{equation}
where the integral converges with respect to the norm in $L_2\bigl((0,\infty); {dy \over y^2}\bigr)$. This integral transform is a direct generalization of the Kontorovich-Lebedev transform, which is obtained by letting $\alpha = 0$ (and doing some simple manipulations).

The \textit{Mehler-Fock transform} is, by definition, given by \cite{nasim1984, yakubovich1996, yakubovichgraaf1999}
\begin{equation} \label{eq:prel_mf_transf}
P_{\mu}[f](\tau) = \int_1^\infty P_{-{1\over 2} + i \tau}^{-\mu}(y) f(y) \, dy
\end{equation}
where $P_{-{1\over 2} + i \tau}^{-\mu}(x)$ is the associated Legendre function of the first kind and, in the general case, $\mu$ is any  complex parameter. The ordinary Mehler-Fock transform corresponds to the case $\mu = 0$; for other values of $\mu$, this integral transform is known as the generalized Mehler-Fock transform. The inverse operator is
\begin{equation} \label{eq:prel_mf_inverse}
f(y) = {1 \over \pi} \int_0^\infty \tau \sinh(\pi\tau) |\Gamma(\tfrac{1}{2} + \mu + i\tau)|^2 P_{-{1 \over 2} + i \tau}^{-\mu}(x) P_{\mu}[f](\tau)\, d\tau
\end{equation}
and (under suitable conditions on $f$ and $\mu$) the Parseval equality
\[
\int_1^\infty |f(x)|^2 dx = {1 \over \pi} \int_0^\infty \tau \sinh(\pi\tau) \bigl|\Gamma(\tfrac{1}{2} + \mu + i \tau)\bigr|^2 |P_{\mu}[f](\tau)|^{2\,} d\tau
\]
holds.

Let us recall some basic facts concerning the special functions involved in the integrals above, which will be useful in the later sections. The modified Bessel functions $I_\nu(x)$ and $K_\nu(x)$ ($x\in(0,\infty), \nu \in \mathbb{C}$) are solutions of the modified Bessel equation $x^2{\partial^2 u \over \partial x^2} + x {\partial u \over \partial x} - (x^2 + \nu^2) u = 0$, and for fixed $x$ they are entire functions of the index $\nu$. The modified Bessel function of the second kind, $K_\nu(x)$, is even with respect to the index, i.e., $K_{-\nu}(x) = K_{\nu}(x)$. For fixed $\nu$ such that $\mathrm{Re}\, \nu > 0$, the asymptotic behavior near zero and infinity is (\cite{dlmf}, \S10.25 and \S10.30)
\begin{align}
\label{eq:besselIK_asymp_zero} I_\nu(x) \sim {x^\nu \over 2^\nu \Gamma(\nu+1)}, \qquad K_\nu(x) \sim {\Gamma(\nu)\, x^{-\nu} \over 2^{\nu+1}}, \qquad x \to 0, \\
\label{eq:besselIK_asymp_infty} I_\nu(x) \sim {e^x \over \sqrt{2\pi x}}, \qquad K_\nu(x) \sim \sqrt{\pi \over 2x}\, e^{-x}, \qquad x \to \infty.
\end{align}
Their Wronskian is (\cite{dlmf}, \S10.28)
\begin{equation} \label{eq:besselIK_wronsk}
\mathcal{W}\{K_\nu(x), I_\nu(x)\} = {1 \over x}
\end{equation}
and the following identity holds (\cite{dlmf}, Eq.\ 10.27.2):
\begin{equation} \label{eq:besselIK_conn}
I_{\nu}(x) = I_{-\nu}(x) - {2 \over \pi} \sin(\nu \pi) K_\nu(x).
\end{equation}

The Whittaker (confluent hypergeometric) functions $M_{\alpha, \eta}(x)$ and $W_{\alpha, \eta}(x)$ ($x \in (0, \infty)$, $\alpha \in \mathbb{C}$,\linebreak $\eta \in \mathbb{C} \setminus \{-{1 \over 2}, -1, -{3 \over 2}, \ldots\}$) are a pair of solutions of Whittaker's equation ${\partial^2 u \over \partial x^2} + (-{1 \over 4} + {\alpha \over x} + {1/4 - \eta^2 \over x^2})u = 0$. The Whittaker functions are, for fixed $x$, analytic functions of $\alpha \in \mathbb{C}$ and $\eta \in \mathbb{C} \setminus \{-{1 \over 2}, -1, -{3 \over 2}, \ldots\}$. The Whittaker $W$ function is even with respect to the second index, $W_{\alpha, \eta}(x) = W_{\alpha, -\eta}(x)$, and it reduces to the modified Bessel function of the second kind when the first index equals zero (\cite{dlmf}, Eq.\ 13.18.9):
\begin{equation} \label{eq:whittaker_besselred}
W_{0,\nu}(2x) = \sqrt{2x \over \pi} K_\nu(x).
\end{equation}
The asymptotic behavior as $x \to 0$ and as $x \to \infty$ is (\cite{dlmf}, \S13.14)
\begin{equation} \label{eq:whittaker_asymp_zero}
\begin{gathered}
M_{\alpha, \eta} (x) = x^{\eta + {1 \over 2}} \bigl(1 + O(x)\bigr)\\[2pt]
W_{\alpha, \eta}(x) = {\Gamma(2\eta) \over \Gamma({1 \over 2} + \eta - \alpha)} x^{{1 \over 2} - \eta} + {\Gamma(-2\eta) \over \Gamma({1 \over 2} - \eta - \alpha)} x^{{1 \over 2} + \eta} + O\bigl(x^{{3 \over 2} - \mathrm{Re}\, \eta}\bigr)
\end{gathered} \qquad\;\; x \to 0 
\end{equation}
(the latter expression is valid for $0 \leq \mathrm{Re}\,\eta < {1 \over 2}$, $\eta \neq 0$) and
\begin{equation} \label{eq:whittaker_asymp_infty}
\begin{gathered}
M_{\alpha, \eta} (x) \sim {\Gamma(1+2\eta) \over \Gamma ({1 \over 2}+\eta-\alpha)}\, x^{-\alpha} e^{{1 \over 2}x}, \qquad W_{\alpha, \eta}(x) \sim x^{\alpha} e^{-{1 \over 2}x},
\end{gathered} \qquad\;\; x \to \infty.
\end{equation}
The Wronksian of the Whittaker functions is (\cite{dlmf}, \S13.14)
\begin{equation} \label{eq:whittaker_wronsk}
\mathcal{W}\{M_{\alpha,-\eta}(x), W_{\alpha,\eta}(x)\} = - {\Gamma(1-2\eta) \over \Gamma({1 \over 2} - \eta - \alpha)}
\end{equation}
and for $2\eta \in \mathbb{C} \setminus \mathbb{Z}$ the Whittaker functions are related as follows (\cite{dlmf}, Eq.\ 13.14.33):
\begin{equation} \label{eq:whittaker_conn}
W_{\alpha,\eta}(x) = {\Gamma(-2\eta) \over \Gamma({1 \over 2} - \eta - \alpha)} M_{\alpha, \eta}(x) + {\Gamma(2\eta) \over \Gamma({1 \over 2} + \eta - \alpha)} M_{\alpha, -\eta}(x).
\end{equation}

The associated Legendre function of the first kind $P_\nu^{-\mu}(x)$ and Olver's associated Legendre function $\bsy{Q}_\nu^\mu(x)$ ($x \in (1,\infty)$, $\nu, \mu \in \mathbb{C}$) constitute a standard pair of solutions of the associated Legendre equation $(x^2-1){\partial^2 u \over \partial x^2} + 2x {\partial u \over \partial x} - \bigl({\mu^2 \over x^2 - 1} + \nu(\nu+1)\bigr)u = 0$. For fixed $x$, the functions $P_\nu^\mu(x)$ and $\bsy{Q}_\nu^\mu(x)$ are entire functions of each parameter $\nu$ and $\mu$. The associated Legendre function of the first kind has the evenness property (\cite{dlmf}, Eq.\ 14.9.11)
\begin{equation} \label{eq:legendreP_even}
P_{{1 \over 2} + \nu}^{-\mu}(x) = P_{{1 \over 2} - \nu}^{-\mu}(x)
\end{equation}
and its derivative satisfies (\cite{dlmf}, Eq.\ 14.10.4)
\begin{equation} \label{eq:legendreP_deriv}
(x^2 - 1){d P_\nu^{-\mu}(x) \over dx} = -(\nu + \mu + 1) P_{\nu+1}^{-\mu}(x) - (\nu + 1) x P_\nu^{-\mu}(x)
\end{equation}
Their asymptotic behavior near the boundary $x=1$ is (\cite{dlmf}, \S14.8(ii))
\begin{equation} \label{eq:legendre_asymp_one}
\begin{aligned}
P_\nu^{-\mu}(x) & \sim {2^{-{\mu \over 2}} \over \Gamma(1 + \mu)} (x-1)^{\mu \over 2} & (-\mu \notin \mathbb{N}) \\
\bsy{Q}_\nu^0(x) & = -{\log(x-1) \over 2 \Gamma(\nu + 1)} + {{1 \over 2} \log 2 - \gamma - \psi(\nu + 1) \over \Gamma(\nu + 1)} + O(x-1) \hspace{-.1\linewidth} & (-\nu \notin \mathbb{N}) \\
\bsy{Q}_\nu^{\mu}(x) & \sim {2^{{\mu \over 2}-1\,} \Gamma(\mu) \over \Gamma(\nu + \mu + 1)} (x-1)^{-{\mu \over 2}} & (\mathrm{Re}\,\mu > 0,\, -\nu - \mu \notin \mathbb{N})
\end{aligned} \qquad\;\; x \to 1
\end{equation}
(here $\psi(\cdot)$ denotes the logarithmic derivative of the Gamma function and $\gamma = -\psi(1)$, cf.\ \cite{lebedev1965}) while near infinity we have (\cite{dlmf}, \S14.8(iii))
\begin{equation} \label{eq:legendre_asymp_infty}
\begin{aligned}
P_\nu^{-\mu}(x) & \sim {\Gamma(\nu + {1 \over 2}) \over \pi^{1 \over 2} \Gamma(\nu + \mu + 1)} (2x)^\nu \quad & (\mathrm{Re}\, \nu > -\tfrac{1}{2},\, -\nu - \mu \notin \mathbb{N})\\
\bsy{Q}_\nu^\mu(x) & \sim {\pi^{1 \over 2} \over \Gamma(\nu + {3 \over 2})} (2x)^{-\nu-1} & (-\nu - \tfrac{1}{2} \notin \mathbb{N})
\end{aligned} \qquad \;\; x \to \infty.
\end{equation}
The Wronskian is (\cite{dlmf}, Eq.\ 14.2.8)
\begin{equation} \label{eq:legendre_wronsk}
\mathcal{W}\{P_\nu^{-\mu}(x), \bsy{Q}_\nu^\mu(x)\} = -{1 \over \Gamma(\nu + \mu + 1) (x^2 - 1)}
\end{equation}
and the following connection formula holds (\cite{dlmf}, Eq.\ 14.9.12):
\begin{equation} \label{eq:legendre_conn}
\cos(\pi\nu) P_\nu^{-\mu}(x) = -{\bsy{Q}_\nu^\mu(x) \over \Gamma(\mu-\nu)} + {\bsy{Q}_{-\nu-1}^\mu(x) \over \Gamma(\nu + \mu + 1)}.
\end{equation}
Finally, we note that in the special case $\mu = {1 \over 2}$, the associated Legendre function of the first kind reduces to an elementary function (\cite{dlmf}, Eq.\ 14.5.16):
\begin{equation} \label{eq:legendreP_reduction}
P_{\nu}^{-1/2}(\cosh \xi) = \biggl({2 \over \pi \sinh\xi}\biggr)^{\!\!1/2\,} {\sinh\bigl(({1 \over 2} + \nu)\xi\bigr) \over {1 \over 2} + \nu}.
\end{equation}

\chapter{Spectral theory of Sturm-Liouville differential operators} \label{chap:spectral}

Historically, many different approaches have been proposed for the investigation of the properties of index transforms (for an overview, see \cite{koornwinder1984}, pp.\ 10-11). In this section we will see that the spectral analysis of Sturm-Liouville differential operators is a useful approach to the study of index transforms, as it provides us with a general method for deriving some of the fundamental properties of these transforms and the associated heat kernels. We begin by summarizing some general results from the spectral theory of Sturm-Liouville differential operators; we will then turn our attention to some specific operators which give rise to index-type transforms, and we finish this section with the introduction and investigation of the generalized Yor integral.

\section{General results} \label{sec:spectral_genresults}

Let $\mathcal{L}$ be a Sturm-Liouville second-order linear differential operator of the form
\begin{equation} \label{eq:sturmliouv_L}
\mathcal{L} = {1 \over r(x)} \biggl[ -{d \over dx}\biggl(p(x) {d \over dx}\biggr) + q(x) \biggr]
\end{equation}
where we will assume that $x \in (a,b)$ with $-\infty \leq a < b \leq +\infty$, $p(x)$ and $r(x)$ are real-valued two times continuously differentiable functions on $(a,b)$, $q(x)$ is a real-valued locally H\"{o}lder continuous function on $(a,b)$, and
\[
p(x),\, r(x) > 0, \qquad q(x) \geq 0 \qquad\qquad (a < x < b).
\]
(These assumptions are made for simplicity, but the method is valid in greater generality; in particular, differential operators of any order $n \in \mathbb{N}$ may be considered. See for instance \cite{dunfordschwartz1963}, Section XIII.5.)

Let $\mathrm{AC}_{\mathrm{loc}}(a,b)$ be the space of real-valued functions on $(a,b)$ which can be written as the integral of a locally integrable function. Then $\mathcal{L}$ is a linear operator on $L_2\bigl((a,b); r(x)dx\bigr)$ with domain
\[
\mathcal{D}(\mathcal{L}) = \bigl\{u \in L_2\bigl((a,b); r(x)dx\bigr): u, u' \in \mathrm{AC}_{\mathrm{loc}}(a,b) \text{ and } \mathcal{L}u \in L_2\bigl((a,b); r(x)dx\bigr) \bigr\}.
\]
It is known (\cite{mckean1956}, \S3) that either the operator $\mathcal{L}$ with domain $\mathcal{D}(\mathcal{L})$ is self-adjoint, or else it becomes a self-adjoint operator if we restrict the domain $\mathcal{D}(\mathcal{L})$ by imposing a suitable boundary condition at $a$, at $b$ or at both $a$ and $b$. In order to determine if boundary conditions must be imposed and deduce the form of the appropriate boundary conditions, one can employ Feller's boundary classification, which we may summarize as follows (cf.\ \cite{mckean1956}, \S3, and \cite{linetsky2006}, pp.\ 237-242): define, for $a < x < y < b$,
\[
\mathcal{S}[x,y] = \int_x^y {1 \over p(t)}\, dt
\]
and consider the integrals (where $c \in (a,b)$ is fixed)
\begin{align*}
& I_a = \int_a^c \biggl(\lim_{x \searrow a}\mathcal{S}[x,y]\biggr) \bigl(1+q(y)\bigr)\, r(y)dy, \; && \; I_b = \int_c^b \biggl(\lim_{y \nearrow b}\mathcal{S}[x,y]\biggr) \bigl(1+q(x)\bigr)\, r(x)dx, \\
& J_a = \int_a^c \mathcal{S}[x,c]\, \bigl(1+q(x)\bigr)\, r(x)dx, \; && \; J_b = \int_c^b \mathcal{S}[c,y]\, \bigl(1+q(y)\bigr)\, r(y)dy.
\end{align*}
The boundary $e \in \{a,b\}$ is said to be: \textit{regular} if $I_e < \infty$ and $J_e < \infty$; \textit{exit} if $I_e < \infty$ and $J_e = \infty$; \textit{entrance} if $I_e = \infty$ and $J_e < \infty$; or \textit{natural} if $I_e = \infty$ and $J_e = \infty$. Then the boundary conditions at $e$ that should be imposed to the functions $u \in \mathcal{D}(\mathcal{L})$ for the operator to become self-adjoint are:
\begin{equation} \label{eq:feller_boundarycond}
\begin{aligned}
(1-\alpha_e) \Bigl(\lim_{x \to e} u(x)\Bigr) + \alpha_e \Bigl(\lim_{x \to e} p(x) u'(x)\Bigr) = 0, \qquad & e \text{ is regular}; \\
\lim_{x \to e} u(x) = 0, \qquad & e \text{ is exit}; \\
\lim_{x \to e} p(x) u'(x) = 0, \qquad & e \text{ is entrance, or } e \text{ is natural and } \biggl|\int_e^c r(x)dx\biggr| < \infty; \\
\text{no boundary condition at } e, \qquad & e \text{ is natural and } \biggl|\int_e^c r(x)dx\biggr| = \infty.
\end{aligned}
\end{equation}
(In the regular case, $\alpha_e \in [0,1]$ can be chosen arbitrarily.)

From now on, let us assume that the boundary conditions \eqref{eq:feller_boundarycond} have been imposed to $\mathcal{D}(\mathcal{L})$, so that the resulting linear operator is self-adjoint. In these conditions, the spectrum $\sigma(\mathcal{L})$ of the self-adjoint operator $\mathcal{L}$ is contained on the interval $[0,\infty)$ (\cite{mckean1956}, \S3). The next theorem shows that the solutions to the second-order differential equation $\mathcal{L} u = \lambda u$ define an integral transform (which is sometimes called the \textit{generalized Fourier transform}) whose inverse is also of the integral transform type.

\begin{theorem} \label{thm:genfourier}
\emph{(\cite{dunfordschwartz1963}, XIII.5.13-14; \cite{weidmann1987}, Theorem 8.7)}
Let $w_1(x,\lambda), w_2(x,\lambda)$ be two continuous functions on $(a,b) \times \mathbb{R}$ such that for any fixed $\lambda \in \mathbb{R}$, $\{w_1(\cdot, \lambda), w_2(\cdot, \lambda)\}$ forms a basis for the space of solutions of $\mathcal{L} u = \lambda u$. Then there exists a positive $2 \times 2$ matrix measure $\rho$ on the nonnegative real axis such that the operator $\mathcal{F}$ given by
\begin{equation} \label{eq:genfourier_tr}
(\mathcal{F} f)_i (\lambda) = \int_a^b f(x)\, \overline{w_i(x,\lambda)}\, r(x)dx
\end{equation}
is well-defined for each $f \in L_2\bigl((a,b); r(x)dx\bigr)$ and defines an isometric isomorphism of $L_2\bigl((a,b); r(x)dx\bigr)$ onto $L_2\bigl([0,\infty); \rho\bigr)$. The inverse is given by
\begin{equation} \label{eq:genfourier_inv}
\bigl(\mathcal{F}^{-1} \widetilde{f}\bigr)(x) = \int_{0-}^\infty \sum_{i,j=1}^2 \widetilde{f}_i(\lambda) w_j(x,\lambda)\, d\rho_{ij}(\lambda), \qquad \widetilde{f} = \bigl(\widetilde{f}_1, \widetilde{f}_2\bigr) \in L_2\bigl([0,\infty); \rho\bigr).
\end{equation}
The convergence of the integrals \eqref{eq:genfourier_tr} and  \eqref{eq:genfourier_inv} is understood with respect to the norm of the spaces $L_2\bigl((a,b); r(x)dx\bigr)$ and $L_2\bigl([0,\infty); \rho\bigr)$ respectively.
\end{theorem}

\begin{remark}
We point out that, in the case where neither of the two boundaries are natural, it can be shown (\cite{mckean1956}, Theorem 3.1) that $\sigma(\mathcal{L})$ is a countable set, the measures $\rho_{ij}$ are discrete, and consequently the inversion integral \eqref{eq:genfourier_inv} reduces to a countable sum. On the other hand, as we shall see in Subsection \ref{sec:index_examples}, there are many important cases (with at least one natural boundary) where the measures $\rho_{ij}$ are absolutely continuous with respect to the Lebesgue measure and the integral \eqref{eq:prel_kl_inverse} is of the same type as the inversion integrals of the classical integral transforms. (In the presence of natural boundaries, the spectrum $\sigma(\mathcal{L})$ and the measures $\rho_{ij}$ may have both a discrete and an absolutely continuous component; see e.g.\ \cite{linetsky2006}, Section 3.4 for further background.)
\end{remark}

To determine the matrix measure $\rho$ and complete the description of the generalized Fourier transform given above, we can use a complex variable technique which relies on a general result concerning the resolvent kernel of the self-adjoint operator $\mathcal{L}$. Under the above assumptions, it is known (\cite{dunfordschwartz1963}, XIII.3.16; \cite{weidmann1987}, Theorem 7.8) that, for each $\lambda \in \mathbb{C} \setminus \sigma(\mathcal{L})$, the equation $\mathcal{L}u = \lambda u$ has, up to a multiplicative constant, a unique solution square-integrable at $a$ with respect to the measure $r(x) dx$ and satisfying the boundary condition at $a$, as well as a unique solution square-integrable at $b$ with respect to the measure $r(x) dx$ and satisfying the boundary condition at $b$. We denote these solutions as $\phi_1(x,\lambda)$ and $\phi_2(x,\lambda)$ respectively. Then, for any $\lambda \in \mathbb{C} \setminus \sigma(\mathcal{L})$, the resolvent $R(\lambda, \mathcal{L}) = (\mathcal{L} - \lambda )^{-1}$ of the self-adjoint operator $\mathcal{L}$ is given by (\cite{dunfordschwartz1963}, XIII.3.16; \cite{weidmann1987}, Theorem 7.8)
\begin{equation} \label{eq:resolvent_phi}
R(\lambda, \mathcal{L}) g(x) = \int_a^b \mathcal{K}_r(x,y;\lambda) g(y)\, r(y)dy
\end{equation}
where $\mathcal{K}_r(x,y;\lambda)$, the kernel with respect to the measure $r(y)dy$, is given by
\begin{equation} \label{eq:resolvent_kernphi}
\mathcal{K}_r(x,y;\lambda) = 
\begin{cases}
{1 \over \mathcal{W}_p\{\phi_2,\phi_1\}(\lambda)} \phi_1(x,\lambda) \phi_2(y,\lambda), & x < y \\
{1 \over \mathcal{W}_p\{\phi_2,\phi_1\}(\lambda)} \phi_2(x,\lambda) \phi_1(y,\lambda), & x > y 
\end{cases}
\end{equation}
and $\mathcal{W}_p\{f,g\}(\lambda) = (p{\partial f \over \partial x}g - pf{\partial g \over \partial x})(x,\lambda)$ is the generalized Wronskian, which may depend on $\lambda$ but not on $x$ (\cite{weidmann1987}, Theorem 5.1).

\begin{theorem} \label{thm:rhomeasure_tk_formula}
\emph{(\cite{dunfordschwartz1963}, XIII.5.18; \cite{weidmann1987}, Theorem 9.7)} Let $Q$ be a complex neighborhood of some open interval $\Lambda \subseteq \mathbb{R}$. Assume that the functions $w_1(x,\lambda)$, $w_2(x,\lambda)$ from Theorem \ref{thm:genfourier} are continuous on $(a,b) \times Q$, analytically dependent on $\lambda \in Q$, and define, for each fixed $\lambda \in Q$, a basis $\{w_1(\cdot, \lambda), w_2(\cdot, \lambda)\}$ for the space of solutions of $\mathcal{L} u = \lambda u$. Then there exist functions $m_{ij}^\pm(\lambda)$ such that the kernel $\mathcal{K}_r(x,y;\lambda)$ for the resolvent $R(\lambda, \mathcal{L})$ has, for all $\lambda \in Q \setminus\sigma(\mathcal{L})$, the representation
\begin{equation} \label{eq:resolvent_kernrep1}
\mathcal{K}_r(x,y;\lambda) = \begin{cases}
\sum_{i,j=1}^2 m_{ij}^-(\lambda) w_i(x,\lambda) \overline{w_j(y,\overline{\lambda})}, & x < y \\
\sum_{i,j=1}^2 m_{ij}^+(\lambda) w_i(x,\lambda) \overline{w_j(y,\overline{\lambda})}, & x > y.
\end{cases}
\end{equation}
The functions $m_{ij}^\pm(\lambda)$ are analytically dependent on $\lambda \in Q\setminus\sigma(\mathcal{L})$, and given any bounded open interval $(\lambda_1, \lambda_2) \subseteq \Lambda$ we have
\begin{equation} \label{eq:rhomeasure_formula}
\begin{aligned}
\rho_{ij}\bigl((\lambda_1, \lambda_2)\bigr) & = \lim_{\delta \searrow 0} \lim_{\eps \searrow 0} {1 \over 2\pi i} \int_{\lambda_1+\delta}^{\lambda_2-\delta} [m_{ij}^-(\lambda + i \eps) - m_{ij}^-(\lambda - i \eps)]\, d\lambda \\
& = \lim_{\delta \searrow 0} \lim_{\eps \searrow 0} {1 \over 2\pi i} \int_{\lambda_1+\delta}^{\lambda_2-\delta} [m_{ij}^+(\lambda + i \eps) - m_{ij}^+(\lambda - i \eps)]\, d\lambda, & i,j = 1,2
\end{aligned}
\end{equation}
where $\rho$ is the positive matrix measure from Theorem \ref{thm:genfourier}. 
\end{theorem}

\begin{remark}
A common way to choose the basis $\{w_1(\cdot, \lambda), w_2(\cdot, \lambda)\}$ for the space of solutions of $\mathcal{L} u = \lambda u$ is to define these functions via the initial conditions
\begin{equation} \label{eq:basissol_initialcond}
w_1(c,\lambda) = 1, \qquad w_1'(c,\lambda) = 0, \qquad w_2(c,\lambda) = 0, \qquad w_2(c, \lambda) = 1
\end{equation}
where $c$ is some fixed point of the interval $(a,b)$. This choice of basis assures that the functions $w_1(x,\lambda)$ and $w_2(x,\lambda)$ are entire functions of $\lambda \in \mathbb{C}$; in particular, such choice assures that the measure $\rho$ can be computed via Equation \eqref{eq:rhomeasure_formula}. But we emphasize that our presentation of the theory allows for a greater freedom in the choice of basis, and in several cases this turns out to be convenient. (See the discussion in \cite{dunfordschwartz1963}, pp.\ 1347-1349.)
\end{remark}

Let us now turn our attention to the parabolic equation associated with the Sturm-Liouville operator $\mathcal{L}$, i.e., to the parabolic partial differential equation (PDE) ${\partial u \over \partial t} = -\mathcal{L}_x u$ (where the subscript indicates that $\mathcal{L}$ acts on the variable $x$). Recall the following definition:
\begin{definition} \label{def:fundsol_mckean}
(cf.\ \cite{mckean1956}, \S1) The function $p_r(t,x,y)$ ($t > 0$, $x,y \in (a,b)$) is said to be a \textit{fundamental solution} (with respect to the measure $r(x)dx$) for the parabolic equation ${\partial u \over \partial t} = -\mathcal{L}_x u$ on the domain $t > 0$, $x \in (a,b)$ and subject to the boundary conditions imposed on $\mathcal{D}(\mathcal{L})$ if it satisfies the following conditions:\vspace{-2pt}
\begin{enumerate}\setlength{\itemsep}{1pt}
\item[(i)] $p_r(t,\cdot,\cdot)$ is positive and symmetric on $(a,b) \times (a,b)$ for each $t > 0$;
\item[(ii)] The derivatives ${\partial^n \over \partial t^n} p_r(t,\cdot,y)$ satisfy the boundary conditions, as well as the PDE ${\partial^n \over \partial t^n} p_r(t,\cdot,y) = (- \mathcal{L}_x)^n p_r(t,\cdot,y)$ ($t>0$, $y \in (a,b)$, $n \in \mathbb{N}$);
\item[(iii)] $\int_a^b p_r(t,x,y)\, r(y)dy \leq 1$ ($t>0$, $x \in (a,b)$);
\item[(iv)] (Chapman-Kolmogorov relation) $p_r(t+s,x,y) = \int_a^b p_r(t,x,\xi) p_r(s,\xi,y) r(\xi)d\xi$ ($t,s>0$, $x,y\in(a,b)$);
\item[(v)] The operators $(S_t \psi)(x) = \int_a^b p_r(t,x,\xi) \psi(\xi)\, r(\xi)d\xi$ constitute a semigroup $(S_t)_{t>0}$ which maps the space $\mathrm{C}_\mathrm{b}(a,b)$ of bounded continuous functions on $(a,b)$ onto itself, so that the PDE ${\partial^n \over \partial t^n} (S_t \psi)(x) = (- \mathcal{L}_x)^n (S_t \psi)(x)$ holds for each $n \in \mathbb{N}$;
\item[(vi)] For any $\varphi \in \mathrm{C}_\mathrm{b}(a,b)$ such that $\mathcal{L}_x \varphi$ is continuous at a neighborhood of $x \in (a,b)$, we have $(S_t \varphi) (x) = \varphi(x) + t\, (-\mathcal{L}_x)\varphi(x) + o(t)$ as $t \searrow 0$. 
\end{enumerate} 
\end{definition}

The connection between the generalized Fourier transform determined by the Sturm-Liouville differential operator and the fundamental solution of the associated parabolic PDE is given in the following theorem:

\begin{theorem} \label{thm:fundsol_integralrep}
Under the assumptions of Theorem \ref{thm:genfourier}, the kernel \eqref{eq:resolvent_kernphi} of the resolvent $R(\mu, \mathcal{L})$ can be written, for $\mu < 0$, as
\begin{equation} \label{eq:spectralexp_resolventkern}
\mathcal{K}_r(x,y;\mu) = \int_{0-}^\infty {1 \over \lambda - \mu} \sum_{i,j=1}^2 w_i(x,\lambda) w_j(y,\lambda)\, d\rho_{ij}(\lambda), \qquad x, y \in (a,b)
\end{equation}
and a fundamental solution of the parabolic PDE ${\partial u \over \partial t} = - \mathcal{L}_x u$ on the domain $t > 0$, $x \in (a,b)$ is given by
\begin{equation} \label{eq:spectralexp_fundsol}
p_r(t,x,y) = \int_{0-}^\infty e^{-t\lambda} \sum_{i,j=1}^2 w_i(x,\lambda) w_j(y,\lambda)\, d\rho_{ij}(\lambda), \qquad t > 0,\; x, y \in (a,b).
\end{equation}
The integrals \eqref{eq:spectralexp_resolventkern} and \eqref{eq:spectralexp_fundsol} converge with respect to the norm of the space $L_2\bigl((a,b);  r(x) dx\bigr)$ relative to each of the variables $x,y$ if the other variables are held fixed.
\end{theorem}

\begin{proof} 
The representation \eqref{eq:spectralexp_resolventkern} for the kernel of the resolvent is a consequence of Corollary 2 in \cite{naimark1968}, \S21.2. Then, the expansion \eqref{eq:spectralexp_fundsol} follows from the results in \cite{mckean1956}, \S4. (The mentioned results of \cite{naimark1968} and \cite{mckean1956} are stated for the case where the basis $\{w_1(\cdot, \lambda), w_2(\cdot, \lambda)\}$ is chosen as in \eqref{eq:basissol_initialcond}, but the same proofs work in our setting.)
\end{proof}

\section{Index transforms as a particular case} \label{sec:index_examples}

Various index transforms can be obtained as a particular case of the generalized Fourier transform of Theorem \ref{thm:genfourier}, provided that one chooses the coefficients of the Sturm-Liouville operator $\mathcal{L}$ so that the kernel of the index transform is a solution of the differential equation $\mathcal{L} u = \lambda u$. This will be better understood through the following examples, which are devoted to the Kontorovich-Lebedev, index Whittaker and Mehler-Fock transforms.

The deduction of the Kontorovich-Lebedev transform through the spectral theory for Sturm-Liouville differential operators was already given in \cite{titchmarsh1962}, \S4.15 (see also \cite{tuanzayed2002}, Example 7), where a different technique was used for the computation of the spectral matrix. In the example below we present a more direct derivation (in our approach, the spectral matrix can be directly obtained via the formula \eqref{eq:rhomeasure_formula}, so unlike in \cite{titchmarsh1962} we will not resort to a transformation of the modified Bessel differential operator for obtaining the spectral expansion), and discuss the connection with the associated parabolic equation.

\begin{example} \label{exam:kl_transf}
(Kontorovich-Lebedev transform)
Consider the case $p(x) = q(x) = x$ and $r(x) = {1 \over x}$ with $a=0$ and $b = \infty$, i.e.,
\begin{equation} \label{eq:example_kl_Lop}
\mathcal{L} = - x^2{d^2 \over d x^2} - x{d \over d x} + x^2, \qquad 0 < x < \infty.
\end{equation}

We have $\mathcal{S}[x,y] = \int_x^y {1 \over t} dt = \log y - \log x$, and it easily follows that $I_a = I_b = J_a = J_b = \infty$, meaning that both boundaries of \eqref{eq:example_kl_Lop} are natural. Since $\int_0^1 {dx \over x} = \int_1^\infty {dx \over x} = \infty$, no boundary conditions need to be imposed to $\mathcal{D}(\mathcal{L})$.

The functions
\begin{equation} \label{eq:example_kl_w1w2}
w_1(x,\lambda) = K_{i\sqrt{\lambda}}(x), \qquad\quad w_2(x,\lambda) = 2 I_{-i\sqrt{\lambda}}(x) - {2 \over \pi} \sin(i\sqrt{\lambda} \pi) K_{i\sqrt{\lambda}}(x)
\end{equation}
are continuous on $(x,\lambda) \in \mathbb{R}^2$ and, for each fixed $\lambda \in \mathbb{C}$, they constitute a basis for the space of solutions of $\mathcal{L}u = \lambda u$. By Theorem \ref{thm:genfourier}, the generalized Fourier transform
\begin{equation} \label{eq:example_kl_isom1}
\begin{gathered}
\mathcal{F}: L_2\bigl([0,\infty); \tfrac{dx}{x}\bigr) \to L_2\bigl([0,\infty); \rho\bigr) \\[1.5pt]
(\mathcal{F}f)(\lambda) = \int_0^\infty f(x)\, \Bigl( w_1(x,\lambda),\; w_2(x,\lambda) \Bigr) {dx \over x}, \qquad (\mathcal{F}^{-1} \widetilde{f})(x) = \int_{0-}^\infty \widetilde{f}(\lambda) \cdot \Bigl( w_1(x,\lambda),\; w_2(x,\lambda) \Bigr) d\rho(\lambda)
\end{gathered}
\end{equation}
is an isomorphism between the two spaces.

To determine the measure $\rho$, we start by noting that, by virtue of \eqref{eq:besselIK_asymp_zero} and \eqref{eq:besselIK_asymp_infty}, the unique solutions of $\mathcal{L}u = \lambda u$ ($\lambda \in \mathbb{C}$, $\mathrm{Im}\,\lambda \neq 0$) belonging to $L_2\bigl((0,1); {dx \over x}\bigr)$ and $L_2\bigl((1,\infty); {dx \over x}\bigr)$ are $\phi_1(x,\lambda) = I_{-i\sqrt{\lambda}}(x)$ and $\phi_2(x,\lambda) = K_{i\sqrt{\lambda}}(x)$ respectively. (Here, with $\lambda = |\lambda| e^{i\theta}$, $0 \leq \theta < 2\pi$, we take $\sqrt{\lambda} = |\lambda|^{1/2} e^{i\theta/2}$.) By \eqref{eq:besselIK_wronsk} we have $\mathcal{W}_p\{\phi_2, \phi_1\} = 1$, so the resolvent kernel in \eqref{eq:resolvent_kernphi} is
\[
\mathcal{K}_r(x,y; \lambda) = \begin{cases}
I_{-i\sqrt{\lambda}}(x) K_{i\sqrt{\lambda}}(y), & x < y\\
K_{i\sqrt{\lambda}}(x) I_{-i\sqrt{\lambda}}(y), & x > y
\end{cases}.
\]
Note that, for $0 \leq \lambda_1 < \lambda_2 < \infty$, the functions $w_1(x,\lambda)$ and $w_2(x,\lambda)$ defined in \eqref{eq:example_kl_w1w2} are analytically dependent on $\lambda$ belonging to a complex neighborhood of $(\lambda_1,\lambda_2)$. (The points in the positive real axis are branch points for $\sqrt{\lambda}$; however, using \eqref{eq:besselIK_conn} we see that both $K_{i\tau}(e^x)$ and $2 I_{-i\tau}(e^x) - {2 \over \pi} \sin(i\tau \pi) K_{i\tau}(e^x)$ are even with respect to $\tau$, and it follows that $w_1(x,\lambda)$ and $w_2(x,\lambda)$, as well as their derivatives with respect to $\lambda$, are continuous at the points $0 < \lambda_0 < \infty$.) Furthermore, it is easy to check that the functions $m_{ij}^\pm(\lambda)$ in the representation \eqref{eq:resolvent_kernrep1} are given by
\[
m_{11}^+(\lambda) = {1 \over \pi} \sin(i\pi\sqrt{\lambda}), \qquad m_{12}^+(\lambda) = {1 \over 2}, \qquad m_{21}^+(\lambda) = m_{22}^+(\lambda) = 0, \qquad m_{ij}^-(\lambda) = m_{ji}^+(\lambda).
\]
By Theorem \ref{thm:rhomeasure_tk_formula} it follows that, for $0 \leq \lambda_1 < \lambda_2 < \infty$,
\begin{align*}
\rho\bigl((\lambda_1, \lambda_2)\bigr) & = \lim_{\delta \searrow 0} \lim_{\eps \searrow 0} {1 \over 2\pi i} \int_{\lambda_1+\delta}^{\lambda_2-\delta} \begin{pmatrix}
{1 \over \pi}(\sin(i\pi\sqrt{\lambda + i\eps}) - \sin(i\pi\sqrt{\lambda - i\eps})) & 0 \\
0 & 0
\end{pmatrix}\, d\lambda \\
& = {1 \over \pi^2} \int_{\lambda_1}^{\lambda_2} \begin{pmatrix}
\sinh(\pi\sqrt{\lambda}) & 0 \\
0 & 0
\end{pmatrix}\, d\lambda.
\end{align*}

Given that $0$ is not an eigenvalue of $\mathcal{L}$ (because the equation $\mathcal{L}u = 0$ has no nontrivial solutions belonging to $L_2\bigl((0,\infty); {dx \over x}\bigr)$, cf.\ \eqref{eq:besselIK_asymp_zero} and \eqref{eq:besselIK_asymp_infty}), we have that $\rho(\{0\})$ is the zero matrix. (See the remark in pp.\ 1360-1361 of \cite{dunfordschwartz1963}.) Thus $\rho_{11}(\cdot)$ is the only nonzero measure, and letting $\tau = \sqrt{\lambda}$ we conclude that the isomorphism \eqref{eq:example_kl_isom1} reduces to
\begin{equation} \label{eq:exam_kl_final}
\begin{gathered}
\mathcal{F}: L_2\bigl((0,\infty); \tfrac{dx}{x}\bigr) \to L_2\bigl( (0,\infty); \tau\sinh(\pi\tau) d\tau \bigr), \\[1.5pt]
(\mathcal{F}f)(\tau) = \int_0^\infty f(x) K_{i\tau}(x)\, {dx \over x}, \qquad (\mathcal{F}^{-1} \widetilde{f})(x) = {2 \over \pi^2} \int_0^\infty \widetilde{f}(\tau) K_{i\tau}(x)\, \tau \sinh(\pi\tau)\, d\tau.
\end{gathered}
\end{equation}
which is the Kontorovich-Lebedev transform \eqref{eq:prel_kl_transf}, \eqref{eq:prel_kl_inverse}. Moreover, Theorem \ref{thm:fundsol_integralrep} yields that the fundamental solution (with respect to the measure ${dx \over x}$) of the parabolic equation ${\partial u \over \partial t} = x^2 {\partial^2 u \over \partial x^2} + x {\partial u \over \partial x} - x^2 u$ on the domain $x \in (0, \infty)$, $t \geq 0$ is given by
\begin{equation} \label{eq:exam_kl_fundsol}
p_r(t,x,y) = {2 \over \pi^2} \int_{0}^\infty e^{-t\tau^2} K_{i\tau}(x) K_{i\tau}(y)\, \tau\sinh(\pi\tau)\, d\tau.
\end{equation}
The corresponding fundamental solution with respect to the Lebesgue measure, $p(t,x,y) = {1 \over x} p_r(t,x,y)$, has been introduced and studied by the second author in \cite{yakubovich2011}, where it was called the \textit{heat kernel for the Kontorovich-Lebedev transform}. The spectral expansion approach shows that the heat kernel can be defined in a similar way for any generalized Fourier transform associated with a Sturm-Liouville differential operator of the form \eqref{eq:sturmliouv_L}, and that the fundamental solution property extends to the general case.
\end{example}

\begin{example} \label{exam:iw_transf} (Index Whittaker transform) 
It will now be shown that if the example above is generalized by replacing the function $q(x) = x$ by the function $q(x) = {1 \over x}(x-\alpha)^2$ (where $\alpha \in (-\infty, {1 \over 2})$ is a fixed parameter), so that
\begin{equation} \label{eq:example_iw_Lop}
\mathcal{L} = - x^2{d^2 \over d x^2} - x{d \over d x} + (x-\alpha)^2, \qquad 0 < x < \infty.
\end{equation}
then the resulting generalized Fourier transform becomes the index Whittaker transform, which is a family of integral transforms that includes the Kontorovich-Lebedev transform as a particular case. This is a direct generalization, since from \eqref{eq:whittaker_besselred} it follows that for $\alpha = 0$ the expressions below coincide with those of Example \ref{exam:kl_transf}.

As in the previous example, it is easily seen that both boundaries of \eqref{eq:example_iw_Lop} are natural and that no boundary conditions are necessary. A convenient basis for the space of solutions of $\mathcal{L}u = \lambda u$ is
\begin{equation} \label{eq:example_iw_w1w2}
\begin{aligned}
& \,w_1(x,\lambda) = \sqrt{\pi \over 2x} W_{\alpha, i\sqrt{\lambda - \alpha^2}}(2x), \\
& \begin{aligned}
w_2(x,\lambda) = & {1 \over \sqrt{2\pi x}} \biggl[  {2\,\Gamma({1 \over 2}-\alpha -i\sqrt{\lambda - \alpha^2}) \over \Gamma(1-2i\sqrt{\lambda - \alpha^2})} M_{\alpha, -i\sqrt{\lambda - \alpha^2}}(2x)  \\
& \;\; + {\Gamma({1 \over 2}-\alpha +i\sqrt{\lambda - \alpha^2})\, \Gamma({1 \over 2}-\alpha -i\sqrt{\lambda - \alpha^2})  \over \Gamma(1+2i\sqrt{\lambda - \alpha^2})\, \Gamma(-2i\sqrt{\lambda - \alpha^2})} W_{\alpha, i\sqrt{\lambda - \alpha^2}}(2x)  \biggr]
\end{aligned}
\end{aligned}
\end{equation}
To compute the measure $\rho$ for the transform \eqref{eq:example_kl_isom1}, we observe that, according to \eqref{eq:whittaker_asymp_zero} and \eqref{eq:whittaker_asymp_infty}, the solutions of $\mathcal{L}u = \lambda u$ which satisfy the integrability conditions are
\[
\phi_1(x,\lambda) = {\Gamma({1 \over 2}-\alpha -i\sqrt{\lambda - \alpha^2}) \over \Gamma(1-2i\sqrt{\lambda - \alpha^2})} {1 \over \sqrt{2\pi x}} M_{\alpha, -i\sqrt{\lambda - \alpha^2}}(2x), \qquad\;\; \phi_2(x,\lambda) = \sqrt{\pi \over 2x} W_{\alpha, i\sqrt{\lambda - \alpha^2}}(2x)
\]
(same choice of branch of $\sqrt{\cdot}$). From \eqref{eq:whittaker_wronsk} it easily follows that $\mathcal{W}_p\{\phi_2,\phi_1\} = 1$, hence
\[
\mathcal{K}_r(x,y; \lambda) = \begin{cases}
{\Gamma({1 \over 2}-\alpha -i\sqrt{\lambda - \alpha^2}) \over \Gamma(1-2i\sqrt{\lambda - \alpha^2})} {1 \over \sqrt{2x}} M_{\alpha, -i\sqrt{\lambda - \alpha^2}}(2x) {1 \over \sqrt{2y}} W_{\alpha, i\sqrt{\lambda - \alpha^2}}(2y), & x < y\\[5pt]
{\Gamma({1 \over 2}-\alpha -i\sqrt{\lambda - \alpha^2}) \over \Gamma(1-2i\sqrt{\lambda - \alpha^2})} {1 \over \sqrt{2x}} W_{\alpha, i\sqrt{\lambda - \alpha^2}}(2x) {1 \over \sqrt{2y}} M_{\alpha, -i\sqrt{\lambda - \alpha^2}}(2y), & x > y
\end{cases}.
\]
If $\lambda_1 < \lambda_2$ are real numbers with $\alpha^2 \notin (\lambda_1,\lambda_2)$, we have that functions defined in \eqref{eq:example_iw_w1w2} are analytically dependent on $\lambda$ belonging to a complex neighborhood of $(\lambda_1,\lambda_2)$. (The analyticity of $w_2(x,\lambda)$ is due to its evenness with respect to $\tau = \sqrt{\lambda-\alpha^2}$, which follows from \eqref{eq:whittaker_conn}.) The functions defined by \eqref{eq:resolvent_kernrep1} are
\begin{gather*}
m_{11}^+(\lambda) = - {\Gamma({1 \over 2}-\alpha +i\sqrt{\lambda - \alpha^2})\, \Gamma({1 \over 2}-\alpha -i\sqrt{\lambda - \alpha^2})  \over 2\pi \, \Gamma(1+2i\sqrt{\lambda - \alpha^2})\, \Gamma(-2i\sqrt{\lambda - \alpha^2})}, \\
m_{12}^+(\lambda) = {1 \over 2}, \qquad\; m_{21}^+(\lambda) = m_{22}^+(\lambda) = 0, \qquad\; m_{ij}^-(\lambda) = m_{ji}^+(\lambda).
\end{gather*}
Consequently
\begin{align*}
& \! \lim_{\eps \searrow 0} \bigl[m_{12}^-(\lambda + i \eps) - m_{12}^-(\lambda - i \eps)\bigr] = 0, \qquad\;\;\lambda \in \mathbb{R} \setminus \{\tfrac{1}{4}\}, \\
& \!\lim_{\eps \searrow 0} \bigl[m_{11}^-(\lambda + i \eps) - m_{11}^-(\lambda - i \eps)\bigr] = 0, \qquad\;\; \lambda < \alpha^2, \\
& \!\begin{aligned}
\lim_{\eps \searrow 0} \bigl[m_{11}^-(\lambda + i \eps) - m_{11}^-(\lambda - i \eps)\bigr] & = {1 \over 2\pi^2} \Bigl|\Gamma\bigl(\tfrac{1}{2}-\alpha +i\sqrt{\lambda - \alpha^2}\bigr)\Bigr|^2 \!\Bigl(\sin\bigl(2\pi i \sqrt{\lambda - \alpha^2}\bigr) - \sin\bigl(-2\pi i \sqrt{\lambda - \alpha^2}\bigr)\!\Bigr)\\
& = {i \over \pi^2} \sinh\bigl(2\pi \sqrt{\lambda - \alpha^2}\bigr) \Bigl|\Gamma\bigl(\tfrac{1}{2}-\alpha +i\sqrt{\lambda - \alpha^2}\bigr)\Bigr|^2\!, \qquad\;\; \lambda > \alpha^2,
\end{aligned}
\end{align*}
so we have  $\rho\bigl((\lambda_1, \lambda_2)\bigr) = 0$ for $-\infty < \lambda_1 < \lambda_2 < \alpha^2$ and 
\[
\rho\bigl((\lambda_1, \lambda_2)\bigr) = {1 \over 2\pi^3} \int_{\lambda_1}^{\lambda_2} \begin{pmatrix}
\sinh\bigl(2\pi \sqrt{\lambda - \alpha^2}\bigr) \Bigl|\Gamma\bigl(\tfrac{1}{2}-\alpha +i\sqrt{\lambda - \alpha^2}\bigr)\Bigr|^2 & 0 \\
0 & 0
\end{pmatrix}\, d\lambda, \qquad \alpha^2 < \lambda_1 < \lambda_2 < +\infty.
\]
In the same manner as in Example \ref{exam:kl_transf} we can see that $\rho(\{\alpha^2\}) = 0$. After letting $\tau = \sqrt{\lambda - \alpha^2}$, the conclusion is that the generalized Fourier transform associated to the operator \eqref{eq:example_iw_Lop} is the isomorphism
\begin{equation} \label{eq:exam_iw_final}
\begin{aligned}
& \qquad \mathcal{F}: L_2\bigl((0,\infty); \tfrac{dx}{x}\bigr) \to L_2\Bigl( (0,\infty); \tau\sinh(2\pi \tau) \bigl|\Gamma\bigl(\tfrac{1}{2}-\alpha +i\tau\bigr)\bigr|^2 d\tau \Bigr) \\
& (\mathcal{F}f)(\tau) = \sqrt{\pi \over 2} \int_0^\infty f(x) W_{\alpha, i\tau}(2x)\, x^{-3/2} dx \\
& (\mathcal{F}^{-1} \widetilde{f})(x) = {\pi^{-5/2} \over \sqrt{2x}} \int_0^\infty \widetilde{f}(\tau) W_{\alpha, i\tau}(2x)\, \tau\sinh(2\pi \tau) \Bigl|\Gamma\bigl(\tfrac{1}{2}-\alpha +i\tau\bigr)\Bigr|^2 d\tau.
\end{aligned}
\end{equation}
Writing $f(x) = (2x)^{-1/2} g(2x)$, we obtain the index Whittaker transform \eqref{eq:prel_iw_transf}, \eqref{eq:prel_iw_inverse}. In addition, Theorem \ref{thm:fundsol_integralrep} assures that the fundamental solution (with respect to the measure ${dx \over x}$) of the parabolic equation ${\partial u \over \partial t} = x^2 {\partial^2 u \over \partial x^2} + x {\partial u \over \partial x} - (x-\alpha)^2 u$ on the domain $x \in (0, \infty)$, $t \geq 0$ is given by
\[
p_r(t,x,y) = {2 \over \pi^2 \sqrt{xy}} \int_{0}^\infty e^{-t(\tau^2+\alpha^2)} W_{\alpha, i\tau}(2x) W_{\alpha, i\tau}(2y)\,\tau\sinh(2\pi \tau) \Bigl|\Gamma\bigl(\tfrac{1}{2}-\alpha +i\tau\bigr)\Bigr|^2 d\tau.
\]
\end{example}

\begin{remark}
If we instead write $f(x) = \pi^{-{1 \over 2}} 2^{{1 \over 2}-2\alpha} e^{-x} x^{{1 \over 2} - \alpha} g({1 \over 4x})$, we obtain the integral transform
\[
\begin{aligned}
& (\mathcal{G}g)(\tau) = \int_0^\infty g(x)\, x^\alpha e^{1 \over 4x} W_{\alpha, i\tau}\Bigl({1 \over 2x}\Bigr)\, x^{-2\alpha} e^{-{1 \over 2x}} dx \\
& (\mathcal{G}^{-1} \widetilde{g})(x) = {2 \over \pi^2} \int_0^\infty \widetilde{g}(\tau)\, x^\alpha e^{1 \over 4x} W_{\alpha, i\tau}\Bigl({1 \over 2x}\Bigr)\, \tau\sinh(2\pi \tau) \Bigl|\Gamma\bigl(\tfrac{1}{2}-\alpha +i\tau\bigr)\Bigr|^2 d\tau.
\end{aligned}
\]
which is the generalized Fourier transform
\[
\mathcal{G}: L_2\bigl((0,\infty); \tfrac{1}{2} x^{-2\alpha} e^{-{1 \over 2x}} dx\bigr) \to L_2\Bigl( (0,\infty); \tau\sinh(2\pi \tau) \bigl|\Gamma\bigl(\tfrac{1}{2}-\alpha +i\tau\bigr)\bigr|^2 d\tau \Bigr)
\]
determined by the operator
\[
\mathcal{A} = -2x^2 {d^2 \over dx^2} - (4(1-\alpha) x + 1) {d \over dx}, \qquad 0 < x < \infty,
\]
i.e., by the operator \eqref{eq:sturmliouv_L} with $p(x) = x^{2(1-\alpha)} e^{-{1 \over 2x}}$, $q(x) = 0$ and $r(x) = {1 \over 2} x^{-2\alpha} e^{-{1 \over 2x}}$. The fundamental solution (with respect to the measure $\tfrac{1}{2} x^{-2\alpha} e^{-{1 \over 2x}} dx$) of the associated parabolic equation is
\[
q_r(t,x,y) = {4 \over \pi^2}\! \int_{0}^\infty\! e^{-t(2\tau^2 + {1 \over 2}(1-2\alpha)^2)} x^\alpha e^{1 \over 4x} W_{\alpha, i\tau}\Bigl({1 \over 2x}\Bigr) y^\alpha e^{1 \over 4y} W_{\alpha, i\tau}\Bigl({1 \over 2y}\Bigr)\,\tau\sinh(2\pi \tau) \Bigl|\Gamma\bigl(\tfrac{1}{2}-\alpha +i\tau\bigr)\Bigr|^2 d\tau.
\]
This fundamental solution has a wide range of applications in mathematical finance; in particular, it is related to the Asian option pricing problem. See Section 4.6 of \cite{linetsky2006} and the references therein.
\end{remark}

The next example is about the (generalized) Mehler-Fock transform, which is also a particular case of a spectral expansion associated with a Sturm-Liouville differential operator. Again, we present a direct derivation, applying the framework of Subsection \ref{sec:spectral_genresults}. (For another approach, see \cite{neretin2001}.)

\begin{example} (Mehler-Fock transform) \label{exam:mf_transf}
Consider now the case $p(x) = (x^2-1)$, $q(x) = {\mu^2 \over x^2-1}$ and $r(x) = 1$, with $a=1$ and $b = \infty$, i.e.,
\begin{equation} \label{eq:example_mf_Lop}
\mathcal{L} = -{d \over d x} \biggl[(x^2-1)  {d \over d x}\biggr] + {\mu^2 \over x^2-1} = - (x^2-1){d^2 \over d x^2} - 2x {d \over d x} + {\mu^2 \over x^2-1}, \qquad 1 < x < \infty.
\end{equation}
(In what follows we assume that $0 \leq \mu < 1$ is a fixed constant.)

Let us determine Feller's boundary classification for the two endpoints. We have
\[
\mathcal{S}[x,y] = \int_x^y {1 \over z^2-1}\, dz = {1 \over 2}\biggl[ \log\biggl({y-1 \over y+1}\biggr) - \log\biggl({x-1 \over x+1}\biggr) \biggr], \qquad x, y \in (1, \infty)
\]
and therefore 
\begin{gather*}
I_a = I_b = J_b = \infty; \qquad\quad
J_a < \infty \;\, \text{ if } \mu = 0, \qquad J_a = \infty \;\, \text{ if } 0 < \mu < 1.
\end{gather*}
Thus $1$ is an entrance boundary if $\mu = 0$ and a natural boundary if $0 < \mu < 1$, while $\infty$ is a natural boundary in either case. Since $\int_1^2 dx < \infty$ and $\int_2^\infty dx = \infty$, it follows from \eqref{eq:feller_boundarycond} that both for $\mu = 0$ and for $0 < \mu < 1$ the suitable boundary condition at $1$ is
\begin{equation} \label{eq:example_mf_bc}
\lim_{x \searrow 1} {(x^2-1) f'(x)} = 0
\end{equation}
whereas at $\infty$ no boundary condition is necessary.

The functions
\begin{equation} \label{eq:example_mf_w1w2}
w_1(x,\lambda) = P_{-{1\over 2} + i\sqrt{\lambda - {1 \over 4}}}^{-\mu}(x), \qquad\quad w_2(x,\lambda) = \begin{cases}
\bsy{Q}_{-{1\over 2} - i\sqrt{\lambda - {1 \over 4}}}^{\mu}(x), & \: \mathrm{Re}\,\lambda < {1 \over 4} \\
\bsy{Q}_{-{1\over 2} + \sqrt{{1 \over 4} - \lambda}}^{\mu}(x), & \: \mathrm{Re}\,\lambda \geq {1 \over 4} 
\end{cases}
\end{equation}
are continuous on $(x,\lambda) \in \mathbb{R}^2$ and, for each fixed $\lambda \in \mathbb{C}$, they constitute a basis for the space of solutions of $\mathcal{L}u = \lambda u$. (The choice of branch for $\sqrt{\cdot}$ is the same as in the previous examples.) By Theorem \ref{thm:genfourier}, we have an isomorphism
\begin{equation} \label{eq:example_mf_isom1}
\begin{aligned}
& \qquad\quad \mathcal{F}: L_2\bigl((1,\infty); dx\bigr) \to L_2\bigl([0,\infty); \rho\bigr) \\[1pt]
& (\mathcal{F}f)(\lambda) = \int_{1}^\infty f(x)\, \Bigl( P_{-{1\over 2} + i\sqrt{\lambda - {1 \over 4}}}^{-\mu}(x),\;\: w_2(x,\lambda) \Bigr) dx \\
& (\mathcal{F}^{-1} \widetilde{f})(x) = \int_{0-}^\infty \widetilde{f}(\lambda) \cdot \Bigl( P_{-{1\over 2} + i\sqrt{\lambda - {1 \over 4}}}^{-\mu}(x),\;\: w_2(x,\lambda) \Bigr) d\rho(\lambda).
\end{aligned}
\end{equation}

Now, the unique solution of $\mathcal{L}u = \lambda u$ ($\lambda \in \mathbb{C}$, $\mathrm{Im}\,\lambda \neq 0$) belonging to $L_2\bigl((1,2]; dx\bigr)$ and satisfying the boundary condition at $x=1$ is, by \eqref{eq:legendreP_deriv} and \eqref{eq:legendre_asymp_one},
\[
\phi_1(x,\lambda) = P_{-{1\over 2} + i\sqrt{\lambda - {1 \over 4}}}^{-\mu}(x).
\]
The unique solution belonging to $L_2\bigl([2,\infty); dx\bigr)$ is, by \eqref{eq:legendre_asymp_infty},
\[
\phi_2(x,\lambda) = \bsy{Q}_{-{1\over 2} - i\sqrt{\lambda - {1 \over 4}}}^{\mu}(x).
\]
From \eqref{eq:legendre_wronsk}, their generalized Wronskian is $\mathcal{W}_p\{\phi_2, \phi_1\}(\lambda) = \Bigl(\Gamma\bigl({1\over 2} + \mu - i{\sqrt{\lambda - {1 \over 4}}}\,\bigr)\Bigr)^{-1}$, so the kernel of the resolvent \eqref{eq:resolvent_phi} is
\[
K(x,y; \lambda) = \begin{cases}
\Gamma\Bigl({1\over 2} + \mu - i\sqrt{\lambda - {1 \over 4}}\,\Bigr)\, P_{-{1\over 2} + i\sqrt{\lambda - {1 \over 4}}}^{-\mu}(x)\, \bsy{Q}_{-{1\over 2} - i\sqrt{\lambda - {1 \over 4}}}^{\mu}(y), & \;x < y\\
\Gamma\Bigl({1\over 2} + \mu - i\sqrt{\lambda - {1 \over 4}}\,\Bigr)\,\bsy{Q}_{-{1\over 2} - i\sqrt{\lambda - {1 \over 4}}}^{\mu}(x)\, P_{-{1\over 2} + i\sqrt{\lambda - {1 \over 4}}}^{-\mu}(y), & \;x > y
\end{cases}
\]

The functions $w_1(x,\lambda)$ and $w_2(x,\lambda)$ defined in \eqref{eq:example_mf_w1w2} are analytically dependent on $\lambda$ belonging to a complex neighborhood of any open interval $I$ of the real axis with ${1 \over 4} \notin I$. (The analyticity of $w_1$ follows from the evenness property \eqref{eq:legendreP_even}.) For all $\lambda$ with $\mathrm{Im}\,\lambda \neq 0$ we have
\[
\overline{w_1(x,\overline{\lambda})} = w_1(x,\lambda), \qquad\quad \overline{w_2(x,\overline{\lambda})} = \begin{cases}
\bsy{Q}_{-{1\over 2} - i\sqrt{\lambda - {1 \over 4}}}^{\mu}(x), & \: \mathrm{Re}\,\lambda < {1 \over 4} \\
\bsy{Q}_{-{1\over 2} - \sqrt{{1 \over 4} - \lambda}}^{\mu}(x), & \: \mathrm{Re}\,\lambda \geq {1 \over 4} 
\end{cases}
\]
so, using \eqref{eq:legendre_conn}, we deduce that the functions $m_{ij}^\pm(\lambda)$ in the representation \eqref{eq:resolvent_kernrep1} are given by
\begin{gather*}
m_{11}^-(\lambda) = 
\begin{cases}
-\cos\Bigl( \pi\bigl({1 \over 2} - \sqrt{{1 \over 4} - \lambda}\, \bigr) \Bigr) \Gamma\Bigl({1 \over 2} +\mu + \sqrt{{1 \over 4} - \lambda}\,\Bigr) \Gamma\Bigl({1 \over 2} +\mu - \sqrt{{1 \over 4} - \lambda}\,\Bigr), & \mathrm{Re}\,\lambda > {1 \over 4},\, \mathrm{Im}\, \lambda < 0 \\
0, & \text{otherwise}
\end{cases} \\
m_{12}^-(\lambda) =
\begin{cases}
\Gamma\Bigl({1 \over 2} +\mu - \sqrt{{1 \over 4} - \lambda}\,\Bigr), & \: \mathrm{Re}\,\lambda > {1 \over 4} \\
\Gamma\Bigl({1 \over 2} +\mu - i\sqrt{\lambda - {1 \over 4}}\,\Bigr), & \: \mathrm{Re}\,\lambda < {1 \over 4}
\end{cases} \\
m_{21}^-(\lambda) = m_{22}^-(\lambda) = 0, \qquad m_{ij}^+(\lambda) = m_{ji}^-(\lambda).
\end{gather*}

For $\lambda \in ({1 \over 4}, \infty)$ we have
\begin{gather*}
\lim_{\eps \searrow 0} \bigl[m_{11}^-(\lambda + i \eps) - m_{11}^-(\lambda - i \eps)\bigr] = \cos\Bigl( \pi\bigl(\tfrac{1}{2} + i\sqrt{\lambda - \tfrac{1}{4}}\, \bigr) \Bigr) \Bigl|\Gamma\Bigl(\tfrac{1}{2} +\mu + i\sqrt{\lambda - \tfrac{1}{4}}\,\Bigr)\Bigr|^2 \\
\lim_{\eps \searrow 0} \bigl[m_{12}^-(\lambda + i \eps) - m_{12}^-(\lambda - i \eps)\bigr] = 0,
\end{gather*}
therefore Theorem \ref{thm:rhomeasure_tk_formula} assures that, for ${1 \over 4} \leq \lambda_1 < \lambda_2 < \infty$,
\begin{align*}
\rho\bigl((\lambda_1, \lambda_2)\bigr) & = \lim_{\delta \searrow 0} {1 \over 2\pi i} \int_{\lambda_1+\delta}^{\lambda_2-\delta} \!\begin{pmatrix}
\cos\Bigl( \pi\bigl(\tfrac{1}{2} + i\sqrt{\lambda - \tfrac{1}{4}}\, \bigr) \Bigr) \Bigl|\Gamma\Bigl(\tfrac{1}{2} +\mu + i\sqrt{\lambda - \tfrac{1}{4}}\,\Bigr)\Bigr|^2 & 0 \\
0 & 0
\end{pmatrix}\, d\lambda \\
& = {1 \over 2\pi} \int_{\lambda_1}^{\lambda_2} \begin{pmatrix}
\sinh\Bigl(\pi\sqrt{\lambda - \tfrac{1}{4}}\,\Bigr) \Bigl|\Gamma\Bigl(\tfrac{1}{2} +\mu + i\sqrt{\lambda - \tfrac{1}{4}}\,\Bigr)\Bigr|^2 & 0 \\
0 & 0
\end{pmatrix}\, d\lambda.
\end{align*}
For $\lambda \in (-\infty, {1 \over 4})$ we have
\[
\lim_{\eps \searrow 0} \bigl[m_{ij}^-(\lambda + i \eps) - m_{ij}^-(\lambda - i \eps)\bigr] = 0, \qquad i,j = 1,2,
\]
thus $\rho\bigl((\lambda_1, \lambda_2)\bigr) = 0$ for $-\infty < \lambda_1 < \lambda_2 <{1 \over 4}$. In addition, we have $\rho(\{{1 \over 4}\}) = 0$, because ${1 \over 4}$ is not an eigenvalue of $\mathcal{L}$ (indeed, the equation $\mathcal{L}u = {1 \over 4} u$ has no nontrivial solutions belonging to $L_2\bigl((1,\infty); dx\bigr)$ and satisfying \eqref{eq:example_mf_bc}). Thus $\rho_{11}(\cdot)$ is the only nonzero measure, and if we let $\tau = \sqrt{\lambda - {1 \over 4}}$ we deduce that the isomorphism defined by the generalized Fourier transform \eqref{eq:example_mf_isom1} is
\begin{equation} \label{eq:exam_mf_final}
\begin{gathered}
\mathcal{F}: L_2\bigl((1,\infty); dx\bigr) \to L_2\bigl( (0,\infty); \tau \sinh(\pi\tau) \bigl|\Gamma\bigl(\tfrac{1}{2} +\mu + i\tau\bigr)\bigr|^2 d\tau \bigr) \\[1pt]
\begin{aligned}
& (\mathcal{F}f)(\tau) = \int_1^\infty f(x) \, P_{-{1\over 2} + i\tau}^{-\mu}(x)\, dx \\
& (\mathcal{F}^{-1} \widetilde{f})(x) = {1 \over \pi} \int_0^\infty \widetilde{f}(\tau)\, P_{-{1\over 2} + i\tau}^{-\mu}(x)\, \tau \sinh(\pi\tau) \bigl|\Gamma\bigl(\tfrac{1}{2} +\mu + i\tau\bigr)\bigr|^2 d\tau.
\end{aligned}
\end{gathered}
\end{equation}
which is the generalized Mehler-Fock transform \eqref{eq:prel_mf_transf}, \eqref{eq:prel_mf_inverse}. Consequently, Theorem \ref{thm:fundsol_integralrep} assures that the function defined as
\begin{equation} \label{eq:example_mf_fundsol}
p_r(t,x,y) = {1 \over \pi} \int_{0}^\infty e^{-t(\tau^2 + 1/4)} P_{-{1\over 2} + i\tau}^{-\mu}(x) P_{-{1\over 2} + i\tau}^{-\mu}(y)\, \tau \sinh(\pi\tau)\, \Bigl|\Gamma\Bigl(\tfrac{1}{2} +\mu + i\tau\Bigr)\Bigr|^2 d\tau
\end{equation}
is the fundamental solution (with respect to the Lebesgue measure) of the parabolic equation ${\partial u \over \partial t} = (x^2-1){d^2u \over d x^2} + 2x {du \over d x} - {\mu^2 \over x^2-1}u$ on the domain $x \in (1, \infty)$, $t \geq 0$.
\end{example}

\begin{remark} \label{rmk:example_mf_specialcase}
In the special case $\mu = {1 \over 2}$, from the relation \eqref{eq:legendreP_reduction} and the reflection formula for the Gamma function it follows that \eqref{eq:exam_mf_final} becomes 
\begin{align*}
(\mathcal{F}f)(\tau) = {\sqrt{2} \over \sqrt{\pi}\, \tau} \int_0^\infty f(\cosh(\xi)) \sin(\tau \xi)\, \sqrt{\sinh(\xi)} d\xi \\
(\mathcal{F}^{-1}\widetilde{f})(\cosh(\xi)) = \sqrt{2 \over \pi\, \sinh(\xi)} \int_0^\infty \tau \widetilde{f}(\tau) \sin(\tau \xi)\, d\tau,
\end{align*}
an integral transform which is equivalent to the Fourier sine transform. Accordingly, \eqref{eq:example_mf_fundsol} specializes to
\[
p_r(t,\cosh(\xi),\cosh(\chi)) = {2 \over \pi} \bigl(\sinh(\xi)\bigr)^{\!-{1 \over 2}} \bigl(\sinh(\chi)\bigr)^{\!-{1 \over 2}\!} \int_0^\infty e^{-t(\tau^2 + {1 \over 4})} \sin(\tau \xi) \sin(\tau \chi)\, d\tau.
\]
In the latter expression, the integral term is, by \cite{erdelyi1954}, Eq.\ 2.6.7, equal to ${\pi^{1/2} \over 2\,t^{1/2}} \exp\bigl(-{t \over 4} - {\xi^2+\chi^2 \over 4t}\bigr) \sinh({\xi\chi \over 2t})$. This gives us a closed-form expression for the fundamental solution of the parabolic PDE ${\partial u \over \partial t} =$\linebreak $(x^2-1){d^2u \over d x^2} + 2x {du \over d x} - {1 \over 4(x^2-1)}u$ on the domain $x \in (1, \infty)$: \vspace*{-5pt}
\begin{equation} \label{eq:example_mf_closeddensity}
p_r(t, \cosh(\xi), \cosh(\chi)) = {1 \over \sqrt{\pi \tau}} \bigl(\sinh(\xi)\bigr)^{\!-{1 \over 2}} \bigl(\sinh(\chi)\bigr)^{\!-{1 \over 2}\!} \exp\biggl(-{t \over 4} - {\xi^2+\chi^2 \over 4t}\biggr) \sinh\Bigl({\xi\chi \over 2t}\Bigr).
\end{equation}
\end{remark}

\section{Yor-type integrals} \label{sec:yorintegrals}

The so-called \textit{Yor integral} \cite{yakubovich2013} is the elementary integral
\begin{equation} \label{eq:yorintegral}
\theta(t,x) = \frac{x\, e^{\pi^{2\!}/2t}}{ \sqrt {2\pi^3t}}\int_0^\infty \exp\left(- {\xi^2\over 2t}\right)\exp\left(-x\cosh \xi\right)\sinh(\xi) \sin\left({\pi \xi\over t}\right) d\xi, \qquad  t,x > 0.
\end{equation}
(For convenience we use the notation from \cite{yor1992}, pp.\ 42-43.) The importance of this integral arises from its applications to pricing problems in mathematical finance. Indeed, Yor proved \cite{yor1980} that the Hartman-Watson probability distribution, which is fundamental for the pricing of Asian options \cite{gemanyor1992} and in the context of the Hull-White stochastic volatility model (\cite{gulisashvili2012}, Sections 4.6 and 4.7), has a probability density given by ${\theta(t,x) \over I_0(x)}$, where $\theta(t,x)$ is the Yor integral \eqref{eq:yorintegral} and $I_0(x)$ is the modified Bessel function of the first kind. (To be more precise, ${\theta(t,x) \over I_0(x)}\, dt = \eta_x(dt) = p_x(t)\, dt$, where $\eta_x$ is the Hartman-Watson law with parameter $x>0$ and $p_x(\cdot)$ is its probability density function.)

In \cite{yakubovich2013} the second author proved that the Yor integral \eqref{eq:yorintegral} is closely related to the Kontorovich-Lebedev transform, as it can be equivalently written as

\begin{equation} \label{eq:yorintegral_klinv}
\theta(t,x) = {1 \over \pi^2} \int_0^\infty e^{-\tau^2 t/2} K_{i\tau}(x)\, \tau \sinh(\pi\tau) d\tau, \qquad  t,x > 0.
\end{equation}
This means that the Yor integral is the inverse Kontorovich-Lebedev transform \eqref{eq:exam_kl_final} of the function $\widetilde{f}(\tau) = {1 \over 2}e^{-\tau^2 t/2}$. (Observe that $\widetilde{f}(\tau) \in L_2\bigl((0,\infty); \tau\sinh(\pi\tau) d\tau\bigr)$ for each fixed $t>0$.) It is therefore natural to generalize the definition of the Yor integral as follows:

\begin{definition}
Assume that $e^{-t\lambda} \in L_2\bigl([0,\infty); \rho\bigr)$ for each $t > 0$. Then the \textit{generalized Yor integral for the operator $\mathcal{L}$} is the integral
\begin{equation} \label{eq:genyorintegral_def}
\vartheta(t,x) = \int_{0-}^\infty e^{-t\lambda} \sum_{i,j=1}^2 w_j(x,\lambda)\, d\rho_{ij}(\lambda), \qquad t > 0,\;\; a < x < b
\end{equation}
i.e., it is the inverse generalized Fourier transform \eqref{eq:genfourier_inv} of the function $e^{-t\lambda}$. (In general, the convergence is understood with respect to the norm of $L_2\bigl([0,\infty); \rho\bigr)$.)
\end{definition}

Up to normalization, the Yor integral corresponds to the case $\mathcal{L} = - x^2{d^2 \over d x^2} - x{d \over d x} + x^2$, $0 < x < \infty$. (Indeed, from \eqref{eq:exam_kl_final} we get $\theta(t,x) = 2\vartheta({t \over 2},x)$.)

We note that the condition $e^{-t\lambda} \in L_2\bigl([0,\infty); \rho\bigr)$ is superfluous whenever the basis $\{w_1(\cdot, \lambda), w_2(\cdot, \lambda)\}$ is defined via \eqref{eq:basissol_initialcond}. Indeed, in this case it is known (see \cite{naimark1968}, Corollary 3 in \S21.2 and the remarks preceding Theorem 3' in \S21.4) that, for each $\mu \in \mathbb{C}$ with $\mathrm{Im}\, \lambda \neq 0$, we have convergence of the integral
\[
\int_{0-}^\infty {1 \over |\lambda - \mu|^2} \sum_{i,j=1}^2 d\rho_{ij}(\lambda).
\]
Since for any small $\eps > 0$ we have
\[
\int_{0-}^\infty e^{-2t\lambda} \sum_{i,j=1}^2 d\rho_{ij}(\lambda) \leq \int_{0-}^\infty e^{-2\eps\lambda} \sum_{i,j=1}^2 d\rho_{ij}(\lambda) \leq C_\eps \int_{0-}^\infty {1 \over |\lambda - i|^2} \sum_{i,j=1}^2 d\rho_{ij}(\lambda), \qquad t \geq \eps
\]
where $C_\eps$ is a positive constant depending on $\eps$, we see that $e^{-t\lambda} \in L_2\bigl([0,\infty); \rho\bigr)$ for all $t \geq \eps$, hence for all $t > 0$. In particular, this observation shows that there exists a generalized Yor integral for any second-order differential operator $\mathcal{L}$ satisfying the assumptions of Subsection \ref{sec:spectral_genresults}.

For a general basis $\{w_1(\cdot, \lambda), w_2(\cdot, \lambda)\}$ satisfying the assumptions of Theorems \ref{thm:genfourier} and \ref{thm:rhomeasure_tk_formula}, the verification of the condition $e^{-t\lambda} \in L_2\bigl([0,\infty); \rho\bigr)$ only requires the study of the behavior of the measure $d\rho(\lambda)$ as $\lambda \to \infty$. (Recall that the measure $\rho$ is finite on bounded intervals.)

According to \cite{yakubovich2013}, Equation (3.7), the Yor integral \eqref{eq:yorintegral} is a solution of the parabolic PDE ${\partial u \over \partial t} = {x^2 \over 2} {\partial^2 u \over \partial x^2} + {x \over 2} {\partial u \over \partial x} - {x^2 \over 2} u$. By other words, the Yor integral is an inverse Kontorovich-Lebedev transform which yields a solution of the parabolic PDE associated with the corresponding Sturm-Liouville operator \eqref{eq:example_kl_Lop}. The following lemma gives a sufficient condition for extending this property to the generalized Yor integral \eqref{eq:genyorintegral_def}:

\begin{lemma} \label{lem:genyorintegral_pdesol}
Assume that the integrals
\begin{align}
\label{eq:genyorintegral_pdesol_hyp1} \int_{0-}^\infty \lambda^n e^{-t\lambda} \sum_{i,j=1}^2 w_j(x,\lambda)\, d\rho_{ij}(\lambda), \qquad n=0,1,2,\ldots \\
\label{eq:genyorintegral_pdesol_hyp2} \int_{0-}^\infty e^{-t\lambda} \sum_{i,j=1}^2 {\partial^n w_j(x,\lambda) \over \partial x^n}\, d\rho_{ij}(\lambda), \qquad n=0,1,2,\ldots
\end{align}
converge absolutely and uniformly for $t \geq t_0 > 0$ and $x$ in compact intervals of $(a,b)$. Then, for $n \in \mathbb{N}$, the generalized Yor integral is a solution of the PDE ${\partial^n u \over \partial t^n} = (-\mathcal{L}_x)^n u$ on the domain $t > 0$, $a < x < b$.
\end{lemma}

\begin{proof}
Clearly, the convergence assumption in the lemma allows us to interchange the integral and the derivatives; recalling that the $w_j(t,x)$ are solutions of $\mathcal{L}_x u = \lambda u$, we see that
\begin{align*}
{\partial^n \vartheta(t,x) \over \partial t^n} & = \int_{0-}^\infty (-\lambda)^n e^{-t\lambda} \sum_{i,j=1}^2 w_j(x,\lambda)\, d\rho_{ij}(\lambda) \\
& = \int_{0-}^\infty e^{-t\lambda} \sum_{i,j=1}^2 \bigl((-\mathcal{L}_x)^n w_j\bigr)(x,\lambda) \, d\rho_{ij}(\lambda) \\
& = \bigl((-\mathcal{L}_x)^n \vartheta\bigr)(t,x). \qedhere
\end{align*}
\end{proof}

The second author has proved \cite{yakubovich2013} that the Yor integral \eqref{eq:yorintegral} satisfies the evolution equation
\[
\int_0^\infty p_r(\tfrac{t}{2},x,\xi)\, \theta(s,\xi)\, {d\xi \over \xi} = \theta(t+s,x), \qquad t, s > 0,\;\: 0 < x < \infty
\]
where $p_r(t,x,y)$ is the fundamental solution \eqref{eq:exam_kl_fundsol} associated with the operator $\mathcal{L} = - x^2{d^2 \over d x^2} - x{d \over d x} + x^2$. We will now outline how this evolution equation may be extended to generalized Yor integrals. Suppose that the convergence hypothesis of Lemma \ref{lem:genyorintegral_pdesol} holds, and fix $s>0$. On the one hand, the properties (especially condition (v) of Definition \ref{def:fundsol_mckean}) of the fundamental solution \eqref{eq:spectralexp_fundsol} of the PDE ${\partial u \over \partial t} = - \mathcal{L}_x u$ assure that the integral
\[
\int_0^\infty p_r(t,x,\xi)\, \vartheta(s,\xi)\, r(\xi) d\xi, \qquad t, s > 0,\;\: a < x < b
\]
is a solution of ${\partial v \over \partial t} = -\mathcal{L}_x v$ satisfying the initial condition $v(0,x;s) = \vartheta(s,x)$. On the other hand, Lemma \ref{lem:genyorintegral_pdesol} implies that $\vartheta(t+s,x)$ is also a solution of ${\partial v \over \partial t} = -\mathcal{L}_x v$ satisfying the same initial condition. It is therefore natural to take advantage of uniqueness results (such as Theorems \ref{thm:feynmankac} and \ref{thm:feynmankac_v2} below) for the solution of the Cauchy problem in order to prove the equality between $\vartheta(t+s,x)$ and $\int_0^\infty p_r(t,x,\xi)\, \vartheta(s,\xi)\, r(\xi) d\xi$. The connection with diffusion processes is useful for verifying the uniqueness conditions, as we will illustrate in Example \ref{exam:iw_genyor_evoleq}.

\begin{example} (Generalized Yor integral for the Mehler-Fock transform) \label{exam:mf_genyor}
Consider the operator $\mathcal{L} = - (x^2-1){d^2 \over d x^2} - 2x {d \over d x} + {\mu^2 \over x^2-1}$, $1 < x < \infty$, where $0 \leq \mu < 1$ is a fixed parameter. As seen in Example \ref{exam:mf_transf}, its generalized Fourier transform is the Mehler-Fock transform. The associated generalized Yor integral is given by
\begin{equation} \label{eq:genyorintegral_mf}
\vartheta(t,x) = {1 \over \pi} \int_0^\infty e^{-t(\tau^2+{1 \over 4})}\, P_{-{1\over 2} + i\tau}^{-\mu}(x)\, \tau \sinh(\pi\tau) \bigl|\Gamma\bigl(\tfrac{1}{2} +\mu + i\tau\bigr)\bigr|^2 d\tau.
\end{equation}
and it is a solution of the PDE ${\partial^n u \over \partial t^n} = \bigl[(x^2-1){d^2 \over d x^2} + 2x {d \over d x} - {\mu^2 \over x^2-1}\bigr]^n u$ on the domain $t > 0$, $1 < x < \infty$ ($n \in \mathbb{N}$). This follows from Lemma \ref{lem:genyorintegral_pdesol} since the integrals \eqref{eq:genyorintegral_pdesol_hyp1}, \eqref{eq:genyorintegral_pdesol_hyp2} satisfy the required convergence condition. Indeed, by \cite{dlmf}, Eq.\ 14.2.4, the associated Legendre function admits the integral representation 
\begin{equation} \label{eq:legendreass_integralrep}
P_{-{1\over 2} + i\tau}^{-\mu}(x) = \biggl({2\over \pi}\biggr)^{1 \over 2} {\Gamma\bigl(\mu + \tfrac{1}{2}\bigr)\, (x^2-1)^{\mu \over 2} \over \bigl|\Gamma\bigl(\tfrac{1}{2} +\mu + i\tau\bigr)\bigr|^2} \int_0^\infty {\cos(\tau\xi) \over (x+ \cosh \xi)^{\mu+{1 \over 2}}}\, d\xi
\end{equation}
so the absolute convergence of the integral \eqref{eq:genyorintegral_pdesol_hyp1} follows from the inequalities
\begin{align*}
& \int_0^\infty (\tau^2+\tfrac{1}{4})^n\, e^{-t(\tau^2+{1 \over 4})}\, \bigl|P_{-{1\over 2} + i\tau}^{-\mu}(x)\bigr|\, \tau \sinh(\pi\tau) \bigl|\Gamma\bigl(\tfrac{1}{2} +\mu + i\tau\bigr)\bigr|^2 d\tau \\
& \leq \biggl({2\over \pi}\biggr)^{1 \over 2} \Gamma\bigl(\mu + \tfrac{1}{2}\bigr)\, (x^2-1)^{\mu \over 2} \int_0^\infty (\tau^2+\tfrac{1}{4})^n\, e^{-t(\tau^2+{1 \over 4})}\, \int_0^\infty {|\cos(\tau\xi)| \over (x+ \cosh \xi)^{\mu+{1 \over 2}}}\, d\xi \tau \sinh(\pi\tau) d\tau \\
& \leq \biggl({2\over \pi}\biggr)^{1 \over 2} \Gamma\bigl(\mu + \tfrac{1}{2}\bigr)\, (x^2-1)^{\mu \over 2} \int_0^\infty (\tau^2+\tfrac{1}{4})^n\, e^{-t(\tau^2+{1 \over 4})}\, \tau \sinh(\pi\tau) d\tau\, \int_0^\infty (1+ \cosh \xi)^{-\mu-{1 \over 2}}\, d\xi
\end{align*}
where the two integrals in the latter expression converge (uniformly for $t \geq t_0 > 0$). Regarding the integral \eqref{eq:genyorintegral_pdesol_hyp2}, notice that \eqref{eq:legendreass_integralrep} and the product rule yield
\[
{\partial^n \over \partial x^n}\Bigr[P_{-{1\over 2} + i\tau}^{-\mu}(x)\Bigl] = {1 \over \bigl|\Gamma\bigl(\tfrac{1}{2} +\mu + i\tau\bigr)\bigr|^2} \sum_{j=0}^n C_j(\mu)\, {\partial^{n-j} \over \partial x^{n-j}} \bigl[(x^2-1)^{\mu \over 2}\bigr] \int_0^\infty {\cos(\tau\xi) \over (x+ \cosh \xi)^{\mu+{1 \over 2}+j}}\, d\xi
\]
where $C_j(\mu)$ are real constants which depend on $\mu$. (The integrals converge absolutely and uniformly on $x$, so we can differentiate under the integral sign.) Hence
\begin{align*}
& \int_0^\infty e^{-t(\tau^2+{1 \over 4})}\, \biggl|{\partial^n \over \partial x^n}\Bigr[P_{-{1\over 2} + i\tau}^{-\mu}(x)\Bigl]\biggr|\, \tau \sinh(\pi\tau) \bigl|\Gamma\bigl(\tfrac{1}{2} +\mu + i\tau\bigr)\bigr|^2 d\tau \\
& \leq \sum_{j=0}^n \biggl|C_j(\mu)\, {\partial^{n-j} \over \partial x^{n-j}} \bigl[(x^2-1)^{\mu \over 2}\bigr]\biggr| \int_0^\infty e^{-t(\tau^2+{1 \over 4})}\, \int_0^\infty {|\cos(\tau\xi)| \over (x+ \cosh \xi)^{\mu+{1 \over 2}+j}}\, d\xi\, \tau \sinh(\pi\tau) d\tau \\
& \leq \int_0^\infty e^{-t(\tau^2+{1 \over 4})}\, \tau \sinh(\pi\tau) d\tau\,\sum_{j=0}^n \biggl|C_j(\mu)\, {\partial^{n-j} \over \partial x^{n-j}}\bigl[(x^2-1)^{\mu \over 2}\bigr]\biggr| \int_0^\infty (1+ \cosh \xi)^{-\mu-{1 \over 2}-j}\, d\xi
\end{align*}
and it follows that \eqref{eq:genyorintegral_pdesol_hyp2} converges absolutely and uniformly for $t \geq t_0 > 0$, $x$ in compact intervals of $(0,\infty)$.

Interestingly, just like the Yor integral \eqref{eq:yorintegral}, \eqref{eq:yorintegral_klinv}, the generalized Yor integral for the Mehler-Fock transform can be written as an integral involving elementary functions only. To prove this claim, we substitute \eqref{eq:legendreass_integralrep} into \eqref{eq:genyorintegral_mf} and apply Fubini's theorem (the absolute convergence has been proved above), obtaining
\[
\vartheta(t,x) = \biggl({2\over \pi}\biggr)^{1 \over 2} \Gamma\bigl(\mu + \tfrac{1}{2}\bigr)\, (x^2-1)^{\mu \over 2} e^{-{t \over 4}} \int_0^\infty (x+ \cosh \xi)^{-\mu-{1 \over 2}} \int_0^\infty e^{-t\tau^2} \tau \sinh(\pi\tau)\cos(\xi\tau) d\tau \, d\xi.
\]
By \cite{prudnikov1986}, Eq.\ 2.5.41.15, the inside integral equals
\[
{1 \over 2} \int_{-\infty}^\infty e^{-t\tau^2 + \pi\tau} \tau \cos(\xi\tau)\, d\tau = {1 \over 4t} \sqrt{\pi \over t} \exp\Bigl({\pi^2 - \xi^2 \over 4t}\Bigr) \biggl[\pi \cos\Bigl({\pi\xi \over 2t}\Bigr) - \xi \sin\Bigl({\pi\xi \over 2t}\Bigr)\biggr]
\]
and we conclude that
\[
\vartheta(t,x) = (2t)^{-{3 \over 2}} \Gamma\bigl(\mu + \tfrac{1}{2}\bigr)\, (x^2-1)^{\mu \over 2} \exp\Bigl({\pi^2 \over 4t}-{t \over 4}\Bigr) \int_0^\infty {\exp\bigl(-{\xi^2 \over 4t}\bigr) \over (x+ \cosh \xi)^{\mu+{1 \over 2}}} \biggl[\pi \cos\Bigl({\pi\xi \over 2t}\Bigr) - \xi \sin\Bigl({\pi\xi \over 2t}\Bigr)\biggr] \, d\xi
\]
which is a representation of \eqref{eq:genyorintegral_mf} as an integral over elementary functions.
\end{example}

\begin{example} (Generalized Yor integral for the index Whittaker transform)  \label{exam:iw_genyor}
Now consider the operator $\mathcal{L} = - x^2{d^2 \over d x^2} - x{d \over d x} + (x-\alpha)^2$, $0 < x < \infty$, where $-\infty < \alpha < {1 \over 2}$ is a fixed parameter. It was seen in Example \ref{exam:iw_transf} that its generalized Fourier transform is the index Whittaker transform. Its generalized Yor integral is
\begin{equation} \label{eq:genyorintegral_iw}
\vartheta(t,x) = {\pi^{-5/2} \over \sqrt{2x}} \int_0^\infty e^{-t(\tau^2 + \alpha^2)} W_{\alpha, i\tau}(2x)\, \tau\sinh(2\pi \tau) \Bigl|\Gamma\bigl(\tfrac{1}{2}-\alpha +i\tau\bigr)\Bigr|^2 d\tau.
\end{equation}

Again, we can use an integral representation of the special function in the integrand to prove that the integrals \eqref{eq:genyorintegral_pdesol_hyp1}, \eqref{eq:genyorintegral_pdesol_hyp2} converge absolutely and uniformly. According to \cite{dlmf}, Eq.\ 13.16.5, we have
\[
{W_{\alpha, i\tau}(2x) \over \sqrt{2x}} = {2^{-i\tau} x^{i\tau} \over \Gamma\bigl({1\over 2} - \alpha +i\tau\bigr)} \int_1^\infty e^{-x\xi} (\xi-1)^{i\tau-{1\over 2} - \alpha} (\xi+1)^{i\tau-{1\over 2} + \alpha} d\xi
\]
and therefore
\begin{equation} \label{eq:genyorintegral_est1}
\begin{aligned}
& \int_0^\infty (\tau^2 + \alpha^2)^n\, e^{-t(\tau^2 + \alpha^2)} \biggl|{W_{\alpha, i\tau}(2x) \over \sqrt{2x}}\biggr|\, \tau\sinh(2\pi \tau) \Bigl|\Gamma\bigl(\tfrac{1}{2}-\alpha +i\tau\bigr)\Bigr|^2 d\tau \\ 
& \leq \int_0^\infty (\tau^2 + \alpha^2)^n\, e^{-t(\tau^2 + \alpha^2)} \tau\sinh(2\pi \tau) \Bigl|\Gamma\bigl(\tfrac{1}{2}-\alpha +i\tau\bigr)\Bigr| d\tau\, \int_1^\infty e^{-x\xi} (\xi-1)^{-{1\over 2} - \alpha} (\xi+1)^{-{1\over 2} + \alpha} d\xi.
\end{aligned}
\end{equation}
Since (for $\alpha < {1 \over 2}$) we have $\bigl|\Gamma\bigl(\tfrac{1}{2}-\alpha +i\tau\bigr)\bigr| \leq \Gamma\bigl(\tfrac{1}{2}-\alpha \bigr) < \infty$, the integral with respect to $\tau$ converges uniformly for $t \geq t_0 > 0$; it is also clear that the integral with respect to $\xi$ converges uniformly for $x \geq x_0 > 0$. In addition, using the same reasoning as in Example \ref{exam:mf_genyor} we see that
\[
{\partial^n \over \partial x^n} \biggl[{W_{\alpha, i\tau}(2x) \over \sqrt{2x}}\biggr] = {2^{-i\tau} \over \Gamma\bigl({1\over 2} - \alpha +i\tau\bigr)} \sum_{j=0}^n \varphi_{n-j}(\tau)\, x^{i\tau-n+j} \int_1^\infty \xi^j e^{-x\xi} (\xi-1)^{i\tau-{1\over 2} - \alpha} (\xi+1)^{i\tau-{1\over 2} + \alpha} d\xi
\]
where each $\varphi_{n-j}(\tau)$ is a polynomial of degree $n-j$ in the variable $\tau$, with complex coefficients; consequently,
\begin{align*}
& \int_0^\infty e^{-t(\tau^2 + \alpha^2)} \biggl|{\partial^n \over \partial x^n} \biggl[{W_{\alpha, i\tau}(2x) \over \sqrt{2x}}\biggr]\biggr|\, \tau\sinh(2\pi \tau) \Bigl|\Gamma\bigl(\tfrac{1}{2}-\alpha +i\tau\bigr)\Bigr|^2 d\tau \\ 
& \!\leq \sum_{\!j=0\!}^n x^{-n+j}\! \int_0^\infty\!\! e^{-t(\tau^2 + \alpha^2)} |\varphi_{n-j}(\tau)|\tau\sinh(2\pi \tau) \Bigl|\Gamma\bigl(\tfrac{1}{2}-\alpha +i\tau\bigr)\Bigr| d\tau \int_1^\infty\! \xi^j e^{-x\xi} (\xi-1)^{-{1\over 2} - \alpha} (\xi+1)^{-{1\over 2} + \alpha} d\xi.
\end{align*}
All the integrals in the last expression converge uniformly for $t \geq t_0 > 0$ and $x \geq x_0 > 0$. Hence the assumption of Lemma \ref{lem:genyorintegral_pdesol} is satisfied, and this assures that the generalized Yor integral \eqref{eq:genyorintegral_iw} is a solution of the PDE ${\partial^n u \over \partial t^n} = \bigl[x^2{d^2 \over d x^2} + x{d \over d x} - (x-\alpha)^2\bigr]^n u$ on the domain $t > 0$, $0 < x < \infty$ ($n \in \mathbb{N}$).
\end{example}

\chapter{Diffusion processes} \label{chap:diffusion}

We will now introduce the diffusion processes associated with the Sturm-Liouville operators from the previous section. We shall then present some properties which relate these stochastic processes with the generalized Fourier transforms (in particular, the index transforms) and the associated Yor integrals.

Let $\mathcal{L}$ be the Sturm-Liouville differential operator \eqref{eq:sturmliouv_L}, and let $p(x)$, $q(x)$ and $r(x)$ be the functions which define $\mathcal{L}$. Throughout this section, besides the assumptions in Section \ref{chap:spectral}, we will also assume that both boundaries $a$ and $b$ are either entrance or natural.

Consider, on the state space $(a,b)$, the stochastic differential equation
\begin{equation} \label{eq:sde_general}
dX_t = \mu(X_t) dt + \sigma(X_t) dW_t
\end{equation}
where $\mu(x) = {p'(x) \over r(x)}$ and $\sigma(x) = \bigl({2p(x) \over r(x)}\bigr)^{1 \over 2}$ are deterministic functions, and $W = \{W_t\}_{t \geq 0}$ is a standard Brownian motion. (For background on the theory of diffusion processes and stochastic differential equations, we refer to Borodin and Salminen \cite{borodinsalminen2002} and references therein.) Our assumptions on $p(x)$ and $r(x)$ assure (see \cite{kunita1984}, Theorem II.5.2) that for each $x \in (a,b)$, this stochastic differential equation has (up to indistinguishability) a unique solution $\{\Xtx\}_{t \geq 0}$ such that $\Xox = x$ and $\Xtx \in (a,b)$ for all $t \geq 0$. (Since $a$ and $b$ are either entrance or natural boundaries, the explosion time is infinite, cf.\ \cite{borodinsalminen2002}, Section II.1.) The stochastic process $\{\Xtx\}_{t \geq 0}$ is a diffusion process whose infinitesimal generator is (cf.\ \cite{borodinsalminen2002}, Section III.5) the differential operator $-\mathcal{L}^{0}$ obtained from $-\mathcal{L}$ by setting $q(x) \equiv 0$, i.e.,
\[
-\mathcal{L}^0 = {1 \over r(x)} {d \over dx}\biggl(p(x) {d \over dx}\biggr).
\]
Now let $A_t = \int_0^t k(\Xsx)\, ds$, where $k(x) = {q(x) \over r(x)}$. (The function $k(x)$ is nonnegative, thus $\{A_t\}_{t \geq 0}$ is an increasing process.) Let $\mb{e}$ be an exponentially distributed random variable with unit mean and independent of $\{\Xtx\}_{t \geq 0}$, and consider the stopping time $\zeta = \inf\{t \geq 0: A_t > \mb{e}\}$. The \textit{killed process} is defined by
\[
\wXtx = \begin{cases}
\Xtx, & t < \zeta \\
\Delta, & t \geq \zeta
\end{cases}
\]
where $\Delta$ is a point outside of $\mathbb{R}$. The process $\{\wXtx\}_{t \geq 0}$ is also a diffusion process, and its infinitesimal generator is precisely the Sturm-Liouville operator $-\mathcal{L}$ (\cite{borodinsalminen2002}, Section II.4).

Let us state the Feynman-Kac theorem, which provides a stochastic representation formula for the solution of the Cauchy problem for the parabolic PDE ${\partial u \over \partial t} = -\mathcal{L}_x u$:

\begin{theorem} (Feynman-Kac theorem -- \cite{heathschweizer2000}, Section 1) \label{thm:feynmankac}
Let $\psi:(a,b) \to [0,\infty)$ and $g:[0,\infty) \times (a,b) \to [0,\infty)$ be nonnegative locally $\alpha$-H\"{o}lder continuous functions (for some $\alpha > 0$). Assume that the functions $\psi$ and $g$ are bounded. Then the unique classical solution $u \in \mathrm{C}^{1,2}\bigl((0,\infty)\times (a,b)\bigr) \cap \mathrm{C}\bigl([0,\infty) \times (a,b)\bigr)$ of the Cauchy problem
\begin{equation} \label{eq:feynmankac_cauchypb}
\begin{aligned}
{\partial u \over \partial t} + \mathcal{L}_x u =g, && \qquad (t,x) \in (0,\infty) \times (a,b)\\
u(0,x) = \psi(x), && \qquad x \in (a,b)
\end{aligned}
\end{equation}
is given by the expectation
\begin{equation} \label{eq:feynmankac_usol}
u(t,x) = \mathbb{E}\biggl[e^{-\int_0^t k(\Xsx)ds} \psi(\Xtx) + \int_0^t e^{-\int_0^s k(\Xlx)d\ell} g(t-s,\Xsx)\,ds \biggr]
\end{equation}
where $\{\Xtx\}_{t \geq 0}$ is the solution of the stochastic differential equation \eqref{eq:sde_general} such that $\Xox = x$.
\end{theorem}

If the operator $\mathcal{L}$ has domain $(a,b) = \mathbb{R}$, we also have the following version of the Feynman-Kac theorem where the nonnegativity and boundedness assumptions on $\psi$ and $g$ are relaxed:

\begin{theorem} \label{thm:feynmankac_v2} (Feynman-Kac theorem -- \cite{karatzasshreve1991}, Theorem 5.7.6 and Problem 5.7.7)
Assume that the Sturm-Liouville operator \eqref{eq:sturmliouv_L} has domain $\mathbb{R}$. For fixed $T>0$, let $\psi:\mathbb{R} \to \mathbb{R}$ and $g:[0,T] \times \mathbb{R} \to \mathbb{R}$ be continuous functions, for which there exist constants $C > 0$, $k > 1$ such that
\begin{equation} \label{eq:feynmankac_polygrowth1}
|\psi(x)| \leq C (1+|x|^k)\, \text{ for all }\, x \in \mathbb{R}, \qquad\; |g(t,x)| \leq C (1+|x|^k)\, \text{ for all } x \in \mathbb{R},\, t \in [0,T].
\end{equation}
Suppose that the functions $\mu(x)$ and $\sigma(x)$ in \eqref{eq:sde_general} satisfy the growth condition $|\mu(x)|^2 + |\sigma(x)|^2 \leq C (1+x^2)$. Assume that $v: [0,T] \times \mathbb{R} \to \mathbb{R}$ belongs to $\mathrm{C}^{1,2}\bigl((0,T]\times \mathbb{R}\bigr) \cap \mathrm{C}\bigl([0,T] \times \mathbb{R}\bigr)$, solves the Cauchy problem \eqref{eq:feynmankac_cauchypb} and grows polynomially in $x$, i.e.,
\begin{equation} \label{eq:feynmankac_polygrowth2}
\max_{t \in [0,T]} |v(t,x)| \leq M (1+|x|^\ell) \;\; \text{ for some } M > 0,\, \ell \geq 1.
\end{equation}
Then $v(t,x) = u(t,x)$, where $u(t,x)$ is defined in \eqref{eq:feynmankac_usol}; in particular, such a solution is unique.

Moreover, when $\mu(x)$ and $\sigma(x)$ are uniformly bounded on $\mathbb{R}$ we can replace \eqref{eq:feynmankac_polygrowth1}, \eqref{eq:feynmankac_polygrowth2} by exponential growth conditions
\[
|\psi(x)| \leq C e^{\nu |x|^{2-\delta}}, \quad\; |g(t,x)| \leq C e^{\nu |x|^{2-\delta}}, \quad\; \max_{t \in [0,T]} |v(t,x)| \leq M e^{\kappa |x|^{2-\eps}} \quad (C, M, \nu, \kappa,  \delta, \eps > 0).
\]
\end{theorem}

Taking into account the properties (v) and (vi) (in Definition \ref{def:fundsol_mckean}) of the fundamental solution \eqref{eq:spectralexp_fundsol} of the parabolic PDE ${\partial u \over \partial t} = -\mathcal{L}_x u$, a consequence of Theorem \ref{thm:feynmankac} is that for each bounded, nonnegative, locally H\"{o}lder continuous function $\psi: (a,b) \to [0,\infty)$ we have
\begin{equation} \label{eq:feynmankac_conseq}
\int_a^b \psi(y)\, p_r(t,x,y)r(y)dy = \mathbb{E}\Bigl[e^{-\int_0^t k(\Xsx)ds} \psi(\Xtx)\Bigr] = \mathbb{E}\bigl[\psi(\wXtx)\bigr]
\end{equation}
where the last equality is a consequence of the definition of the killed process $\{\wXtx\}_{t \geq 0}$. From the arbitrariness of the function $\psi$ it follows that the fundamental solution $p_r(t,x,y)$ is the transition probability density of the diffusion process $\{\wXtx\}_{t \geq 0}$ with respect to the measure $r(x) dx$.

To illustrate how this leads to interesting applications of index transforms in the characterization of certain diffusion processes, let us consider the additive process $A_t = \int_0^t k(\Xsx) ds$ given $\Xtx = y$. For $\psi: (a,b) \to [0,\infty)$ as given in the Feynman-Kac theorem, we have
\begin{align*}
\mathbb{E}\Bigl[e^{-A_t} \psi(\Xtx)\Bigr] & = \int_a^b \int_0^\infty \psi(y)\, e^{-\xi} P\bigl[A_t \in d\xi, \Xtx \in dy\bigr] \\
& = \int_a^b \psi(y) \int_0^\infty e^{-\xi} P\bigl[A_t \in d\xi \bigm| \Xtx = y\bigr]\, P\bigl[\Xtx \in dy\bigr] \\
& = \int_a^b \psi(y)\, \mathbb{E}\bigl[e^{-A_t} \bigm| \Xtx = y\bigr] \, p_r^0(t,x,y) r(y) dy
\end{align*}
where $p_r^0(t,x,y)$ denotes the fundamental solution of the parabolic PDE ${\partial u \over \partial t} = -\mathcal{L}^0_x u$. Comparing with the left hand side of \eqref{eq:feynmankac_conseq} we conclude that
\begin{equation} \label{eq:fundsol_fk_rel}
p_r(t,x,y) = \mathbb{E}\bigl[e^{-A_t} \bigm| \Xtx = y\bigr] \, p_r^0(t,x,y)
\end{equation}
and employing the representation \eqref{eq:spectralexp_fundsol} we reach an explicit expression for the conditional expectation of $e^{-A_t}$:
\begin{equation} \label{eq:fundsol_fk_relexplicit}
\mathbb{E}\bigl[e^{-A_t} \bigm| \Xtx = y\bigr] = {\int_{0-}^\infty e^{-t\lambda} \sum_{i,j=1}^2 w_i(x,\lambda) w_j(y,\lambda)\, d\rho_{ij}(\lambda) \over \int_{0-}^\infty e^{-t\lambda} \sum_{i,j=1}^2 w_i^0(x,\lambda) w_j^0(y,\lambda)\, d\rho_{ij}^0(\lambda)}
\end{equation}
where the basis $\{w_1^0(\cdot, \lambda), w_2^0(\cdot, \lambda)\}$ and the matrix measure $\rho^0$ are defined via the operator $\mathcal{L}^0$.

Actually, the relation \eqref{eq:fundsol_fk_rel} yields the following monotonicity property for the fundamental solutions:

\begin{proposition} \label{prop:fundsolineq}
For $j = 1,2$, consider the Sturm-Liouville operator $\mathcal{L}_j = {1 \over r(x)} \bigl[ -{d \over dx}\bigl(p(x) {d \over dx}\bigr) + q_j(x) \bigr]$, $x \in (a,b)$, where the coefficient functions $p(x)$, $r(x)$, $q_1(x)$ and $q_2(x)$ satisfy the assumptions of Subsection \ref{sec:spectral_genresults}. Let $p_r^j(t,x,y)$ be the fundamental solution for the parabolic equation ${\partial u \over \partial t} = -(\mathcal{L}_j)_x u$, as defined in Definition \ref{def:fundsol_mckean}. Assume that $q_1(x) \leq q_2(x)$ for all $x \in (a,b)$. Then
\[
p_r^1(t,x,y) \geq p_r^2(t,x,y) \qquad\;\; \text{for all}\;\; t > 0,\; x, y \in (a,b).
\]
\end{proposition}

The intuition behind the following proof is, under the probabilistic viewpoint, quite straightforward: if we denote by $\{\widetilde{X}_{t}^{j,x\smash[t]{\mathstrut}}\}$ the process with infinitesimal generator $\mathcal{L}_j$, the inequality $q_1(x) \leq q_2(x)$ means that the process $\{\widetilde{X}_{t}^{2,x\smash[t]{\mathstrut}}\}$ is killed at a faster rate than the process $\{\widetilde{X}_{t}^{1,x\smash[t]{\mathstrut}}\}$.

\begin{proof}
By \eqref{eq:fundsol_fk_rel}, we have
\begin{equation} \label{eq:fundsolineq_proof1}
{p_r^2(t,x,y) \over p_r^1(t,x,y)} = {\mathbb{E}\bigl[e^{-A_t^2} \bigm| \Xtx = y\bigr] \over \mathbb{E}\bigl[e^{-A_t^1} \bigm| \Xtx = y\bigr]}
\end{equation}
where (for $j=1,2$) $A_t^j = \int_0^t k_j(\Xsx) ds$ and $k_j(x) = {q_j(x) \over r(x)}$.

Fix $t>0$ and $x \in (a,b)$. The assumption $q_1(x) \leq q_2(x)$ implies that almost surely (a.s.) $e^{-A_t^2} \leq e^{-A_t^1}$, and by the properties of conditional expectation (see \cite{shiryaev1996}, Section II.7) it follows that
\[
\mathbb{E}\bigl[e^{-A_t^2} \bigm| \Xtx = y\bigr] \leq \mathbb{E}\bigl[e^{-A_t^1} \bigm| \Xtx = y\bigr] \qquad\;\; \text{a.s.\ with respect to the measure } P[\Xtx \in dy].
\]
We know that $P[\Xtx \in dy] = p_r^0(t,x,y) r(y) dy$ where $p_r^0(t,x,y) r(y)$ is a positive and real-valued function of $y \in (a,b)$. Thus (as a consequence of the Radon-Nikodym theorem, cf.\ \cite{cohn2013}, Exercise 4.2.9) the preceding inequality holds Lebesgue almost everywhere; by \eqref{eq:fundsolineq_proof1}, this means that
\[
{p_r^2(t,x,y) \over p_r^1(t,x,y)} \leq 1 \qquad \text{for almost all } y.
\]
But, according to the properties of the fundamental solution, the left hand side is a continuous function of $y$. The result follows.
\end{proof}

We now look more closely at the diffusion processes generated by the Sturm-Liouville operators which define the index transforms.

\begin{example} (Diffusion processes associated with the Mehler-Fock transform)
Let $\mathcal{L}$ be the operator defined in \eqref{eq:example_mf_Lop}. Then the coefficients of the stochastic differential equation are $\mu(x) = 2x$ and $\sigma(x) = \sqrt{2(x^2-1)}$, and \eqref{eq:sde_general} becomes
\[
dX_t = 2X_t\, dt + \sqrt{2(X_t^2 - 1)}\, dW_t.
\]
Notice that if $\{\Xtx\}$ is the solution of this stochastic differential equation starting at $\Xox = x \in (1,\infty)$, then its infinitesimal generator is the operator $-\mathcal{L}^0 = (x^2-1){d^2 \over d x^2} + 2x {d \over d x}$, which is the operator whose spectral expansion yields the ordinary Mehler-Fock transform.

Take $0 \leq \mu_1 \leq \mu_2 < 1$ and write $q_\mu(x) = {\mu^2 \over x^2 - 1}$. Then $q_{\mu_1}(x) \leq q_{\mu_2}(x)$ for all $x \in (1, \infty)$, and from Proposition \ref{prop:fundsolineq} we can conclude that $p_r^{\mu_1} (t,x,y) \geq p_r^{\mu_2}(t,x,y)$, where $p_r^{\mu_j} (t,x,y)$ is given by the right hand side of \eqref{eq:example_mf_fundsol} with $\mu = \mu_j$. In other words, the fundamental solution (i.e., transition density of the killed process)
\[
p_r^{\mu}(t,x,y) = {1 \over \pi} \int_{0}^\infty e^{-t(\tau^2 + 1/4)} P_{-{1\over 2} + i\tau}^{-\mu}(x) P_{-{1\over 2} + i\tau}^{-\mu}(y)\, \tau \sinh(\pi\tau)\, \Bigl|\Gamma\Bigl(\tfrac{1}{2} +\mu + i\tau\Bigr)\Bigr|^2 d\tau
\]
is a decreasing function of $0 \leq \mu < 1$ (for fixed $t>0$, $x,y \in (1,\infty)$).

It is also interesting to note that, according to Remark \ref{rmk:example_mf_specialcase}, the process $\{\wXtx\}$ corresponding to $\mu = {1 \over 2}$, which is the process $\{\Xtx\}$ killed at time $\zeta = \inf\{t \geq 0:\, {1 \over 4}\! \int_0^t {1 \over (\Xsx)^2-1} ds > \mb{e}  \}$, has a transition density which is given in closed form by \eqref{eq:example_mf_closeddensity}. Note that to obtain this result we did not rely on any explicit expression for the process $\{\Xtx\}$ in terms of the underlying Brownian motion; we only relied on the PDE which is satisfied by the transition density.
\end{example}

\begin{example} (Diffusion processes associated with the Kontorovich-Lebedev and index Whittaker transforms)
When $\mathcal{L}$ is the operator \eqref{eq:example_iw_Lop}, we have $\sigma(x) = \sqrt{2}\, x$ and $\mu(x) = x$. The diffusion process with infinitesimal generator $-\mathcal{L}^0$ is then the solution of $dX_t = X_t dt + \sqrt{2}\, X_t dW_t$. This is the stochastic differential equation which defines the geometric Brownian motion, so we have
\begin{equation} \label{eq:diff_iw_gbm}
\Xtx = x \exp(\sqrt{2}\,W_t)
\end{equation}
and the well-known expression for the transition density of geometric Brownian motion (\cite{borodinsalminen2002}, Eq.\ 9.1.0.6) gives
\[
p_r^0(t,x,y) = {1 \over 2\sqrt{\pi t}} \exp\Bigl(-{1 \over 4t}(\log y - \log x)^2\Bigr).
\]

In the case $\alpha = 0$ (where $\mathcal{L}$ yields the Kontorovich-Lebedev transform), the additive process in the definition of the killing time is given by $A_t = x^2 \int_0^t \exp(2\sqrt{2}\, W_s)\,ds$. The integral $\int_0^t \exp(2\sqrt{2}\, W_s)\,ds$ (and the more general integral $\int_0^t \exp(a\sqrt{2}\, W_s)\,ds$, with $a$ a real constant, which can be reduced to the former) belongs to a family of exponential functionals of Brownian motion which has been extensively studied \cite{yor1992} due to its essential role in the Asian option pricing problem. Using the relation \eqref{eq:fundsol_fk_relexplicit} and recalling the integral representation \eqref{eq:exam_kl_fundsol} for the transition density of the process with infinitesimal generator $\mathcal{L}$, we obtain a closed-form expression for the conditional Laplace transform of this exponential functional:
\begin{equation} \label{eq:diff_kl_laplace_cond}
\mathbb{E}\Bigl[e^{-x^2 \int_0^t \exp(2\sqrt{2}\, W_s)\,ds} \Bigm| \Xtx = y\Bigr] = {4\sqrt{t}\, y \over \pi^{3/2}} \exp\Bigl({1 \over 4t}(\log y - \log x)^2\Bigr) \! \int_{0}^\infty e^{-t\tau^2} K_{i\tau}(x) K_{i\tau}(y)\, \tau\sinh(\pi\tau)\, d\tau.
\end{equation}
In addition, by \eqref{eq:feynmankac_conseq} (with $\psi(x) \equiv 1$) the unconditional Laplace transform is given by
\begin{equation} \label{eq:diff_kl_laplace_uncond1}
\mathbb{E}\Bigl[e^{-x^2 \int_0^t \exp(2\sqrt{2}\, W_s)\,ds} \Bigr] = {2 \over \pi^2} \int_0^\infty \int_{0}^\infty e^{-t\tau^2} K_{i\tau}(x) K_{i\tau}(y)\, \tau\sinh(\pi\tau)\, d\tau {dy \over y}
\end{equation}
But due to the properties of Brownian motion we also have that the Laplace transform of this functional of geometric Brownian motion can be explicitly written through the elementary integral 
\begin{equation} \label{eq:diff_kl_laplace_uncond2}
\mathbb{E}\Bigl[e^{-x^2 \int_0^t \exp(2\sqrt{2}\, W_s)\,ds} \Bigr] = \mathbb{E}\Bigl[e^{-{x^2 \over 2} \int_0^{2t} \exp(2W_s)\,ds} \Bigr] = {1 \over 2\sqrt{\pi t}} \int_{-\infty}^\infty e^{ix \sinh(y)} e^{-{y^2 \over 4t}} dy
\end{equation}
(the first equality is due to the scaling property of Brownian motion; the second equality is a consequence of Bougerol's identity for Brownian motion, cf.\ \cite{gulisashvili2012}, Section 4.2 and Eq.\ (4.8)). Combining \eqref{eq:diff_kl_laplace_uncond1} and \eqref{eq:diff_kl_laplace_uncond2} we obtain the following integral identity involving the modified Bessel function $K_{i\tau}(x)$:
\[
{2 \over \pi^2} \int_0^\infty \int_{0}^\infty e^{-t\tau^2} K_{i\tau}(x) K_{i\tau}(y)\, \tau\sinh(\pi\tau)\, d\tau {dy \over y} = {1 \over 2\sqrt{\pi t}} \int_{-\infty}^\infty e^{ix \sinh(y)} e^{-{y^2 \over 4t}} dy.
\]

In the more general case $\alpha \leq 0$ it is possible to derive, by following the same reasoning which led us to \eqref{eq:diff_kl_laplace_cond}, a closed-form expression for the conditional expectation $\mathbb{E}\bigl[e^{-x^2 \int_0^t (\exp(\sqrt{2}\, W_s) + c)^2\,ds} \bigm| \Xtx = y\bigr]$, being $c$ a nonnegative constant. Furthermore, Proposition \ref{prop:fundsolineq} asserts that the fundamental solution
\begin{equation} \label{eq:diff_iw_fundsol}
p_r^\alpha(t,x,y) = {2 \over \pi^2 \sqrt{xy}} \int_{0}^\infty e^{-t(\tau^2+\alpha^2)} W_{\alpha, i\tau}(2x) W_{\alpha, i\tau}(2y)\,\tau\sinh(2\pi \tau) \Bigl|\Gamma\bigl(\tfrac{1}{2}-\alpha +i\tau\bigr)\Bigr|^2 d\tau
\end{equation}
is an increasing function of $\alpha \leq 0$ (for fixed $t > 0$, $x, y \in (0,\infty)$).
\end{example}

We conclude this section with an example which demonstrates that, as discussed in Subsection \ref{sec:yorintegrals}, the connection with diffusion processes is useful for deriving an evolution equation for generalized Yor integrals.

\begin{example} \label{exam:iw_genyor_evoleq}
Consider the generalized Yor integral for the index Whittaker transform with parameter $\alpha < {1 \over 2}$, defined by \eqref{eq:genyorintegral_iw}. In order to show that this generalized Yor integral obeys the evolution equation, we start by estimating $\vartheta(t,x)$: using the inequality $|W_{\alpha, i\tau}(2y)| \leq C\, \bigl|\Gamma({1 \over 2} - \alpha + i\tau)\bigr|^{-2} (2y)^\alpha e^{-y}$ (see \cite{srivastava1998}, Eq.\ (1.15)) we deduce that
\[
|\vartheta(t,x)| \leq C \pi^{-5/2} (2x)^{\alpha - {1 \over 2}} e^{-x} \int_0^\infty e^{-t(\tau^2 + \alpha^2)} \, \tau\sinh(2\pi \tau) d\tau.
\]
We know from Example \ref{exam:iw_genyor} that $\vartheta(t,x)$ is a solution of ${\partial u \over \partial t} = x^2{d^2u \over d x^2} + x{du \over d x} - (x-\alpha)^2 u$ on the domain $t > 0$, $0 < x < \infty$. It easily follows that $\vartheta(t,e^y)$ is a solution of ${\partial v \over \partial t} = {\partial^2 v \over \partial y^2} - (e^y - \alpha)^2 v$ on the domain $t > 0$, $y \in \mathbb{R}$. For fixed $s, T > 0$, the above estimate gives
\[
\max_{t \in [s,\infty)} |\vartheta(t+s,e^y)| \leq M e^{\kappa|y|}
\]
for constants $M, \kappa > 0$ which may depend on $s$. Hence Theorem \ref{thm:feynmankac_v2} (where we take $\mathcal{L} = -{\partial^2 \over \partial y^2} + (e^y - \alpha)^2$, $\psi(y) = \vartheta(s,e^y)$, $g(t,y) = 0$ and $v(t,y) = \vartheta(t+s,e^y)$) implies that
\[
\vartheta(t+s,e^y) = \mathbb{E}\Bigl[e^{-\int_0^t (\exp(\sqrt{2}\, W_s + y) - \alpha)^2 ds}\, \vartheta\bigl(s,\exp(\sqrt{2}\, W_s +y)\bigr)\Bigr]
\]
which, by \eqref{eq:feynmankac_conseq}, \eqref{eq:diff_iw_gbm}, is equivalent to the evolution equation
\[
\vartheta(t+s,x) = \int_0^\infty p_r^\alpha(t,x,\xi)\, \vartheta(s,\xi)\, {d\xi \over \xi}.
\]
where $p_r^\alpha(t,x,\xi)$ is the transition density \eqref{eq:diff_iw_fundsol}.
\end{example}

\chapter*{Acknowledgments}

We thank M.\ Guerra for helpful discussions. The work of both authors was partially supported by CMUP (UID/MAT/00144/2013), which is funded by FCT (Portugal) with national (MEC) and European structural funds through the program FEDER under the partnership agreement PT2020, and Project STRIDE -- NORTE-01-0145-FEDER-000033, funded by ERDF -- NORTE 2020. The first author was also supported by the grant PD/BI/128072/2016, under the FCT PhD Programme UC\textbar UP MATH PhD Program.

\renewcommand{\bibname}{References}

\end{document}